\documentclass{lmcs} 
\pdfoutput=1

\usepackage{lastpage}
\lmcsdoi{17}{2}{23}
\lmcsheading{}{\pageref{LastPage}}{}{}%
{Feb.~27,~2020}{May~28,~2021}{}

\usepackage[utf8]{inputenc}

\keywords{initial algebra, terminal coalgebra, algebraicly complete category, finitary functor}

\theoremstyle{plain} 

\usepackage{latexsym}
\usepackage{amsthm}
\usepackage{amsmath}

\usepackage{amsfonts}
\usepackage{amssymb}
\usepackage{appendix}
\usepackage[mathcal]{euscript}
\usepackage{mathrsfs} 

\newcommand\op{\operatorname{op}}
\newcommand\card{\operatorname{card\,}}
\newcommand\Set{\operatorname{\bf Set}}

\renewcommand\Vec{\operatorname{\bf Vec}}
\newcommand{\id}{\operatorname{id}}
\newcommand{\Id}{\operatorname{Id}}
\newcommand{\Coalg}{\operatorname{\bf Coalg}}

\newcommand\colim{\operatorname{\it colim}}
\newcommand{\Ord}{\operatorname{Ord}}
\newcommand{\Nom}{\operatorname{\bf Nom}}

\newcommand\ck{\mathcal {K}}
\newcommand\ca{\mathcal {A}}
\newcommand\cd{\mathcal {D}} 
\newcommand\cp{\mathcal {P}} 
\newcommand\cf{\mathcal {F}} 

\newcommand\N{\mathbb{N}} 
\newcommand\ba{\mathbb{A}} 

\theoremstyle{defC}
\newtheorem{exasC}[thm]{Examples}
\theoremstyle{thmC}
\newtheorem{corC}[thm]{Corollary}

\input xy
\xyoption{all}

\swapnumbers

\numberwithin{equation}{section}

\begin{document}

\title{Algebraic Cocompleteness and Finitary Functors}

\author[J. Ad{\'{a}}mek]{Ji{\v{r}}{\'{\i}}  Ad{\'{a}}mek$^*$}

\address{Department of Mathematics, Czech Technical University in Prague, Czech Republic
\newline Institute of Theoretical Computer Science, Technical University Braunschweig, Germany} 
\email{j.adamek@tu-bs.de}

\begin{abstract}
\noindent
A number of categories is presented that are  algebraically complete and cocomplete, i.e., every endofunctor has  an initial algebra and a terminal coalgebra. Example: assuming GCH, the category $Set_{\leq \lambda}$ of sets of power at most $\lambda$ has that property, whenever
$\lambda$ is an uncountable regular cardinal.

For all finitary (and, more generally, all precontinuous) set functors  the initial algebra and terminal coalgebra  are proved to carry a canonical partial order with the same ideal completion. And they also both  carry a canonical ultrametric with the same Cauchy completion.
Finally, all endofunctors of the category $\Set_{\leq \lambda}$ are finitary if  $\lambda$ has countable cofinality 
and there are no measurable cardinals $\mu \leq \lambda$.
\end{abstract}

\maketitle

\section{Introduction}

\footnote{$^*$Supported by the Grant Agency of the Czech Republic under the Grant 19-00902S.}
The importance of terminal coalgebras for an endofunctor $F$ was clearly demonstrated by Rutten \cite{R}: for state-based systems whose state-objects lie in a category $\ck$ and whose dynamics are described by $F$, the terminal coalgebra $\nu F$ collects behaviors of individual states. And given a system $A$ the unique coalgebra homomorphism from $A$ to $\nu F$ assigns to every state its behavior. However, not every endofunctor has a terminal  coalgebra. Analogously, an initial algebra $\mu F$ need not exist.

Freyd \cite{F} introduced the concept of an algebraically complete category: this means  that every endofunctor has an initial algebra. More than a decade prior to Freyd's lecture Trnkov\'a proved that the category  $\Set_{\leq \aleph_0}$ of countable sets and mappings is
algebraically complete \cite{T}, Thm.2. This has inspired Koubek and the author to prove that for every cardinal
$\lambda$ the categories
 $$
Set_{\leq \lambda}
$$
of sets of power at most $\lambda$ and 
 $$
K\mbox{-}\Vec_{\leq \lambda}
$$
of vector spaces of dimension at most $\lambda$, for any  field $K$, are algebraically complete \cite{AK-79}, Example 14.
The algebraic completeness of the category of classes and maps has been proved in \cite{AMV}.

 We dualize the above concept, and call a category \textit {algebraically cocomplete} if every endofunctor has a terminal coalgebra. Assuming the generalized continuum hypothesis (GCH), we prove below that for every uncountable, regular cardinal $\lambda$ the category $\Set_{\leq \lambda}$ is algebraically complete and cocomplete. (That is, every endofunctor $F$ has both $\mu F$ and $\nu F$.) In contrast, the category of countable sets is not algebraically cocomplete (Example 2.15). And the algebraic cocompleteness of the category $\Set_{\leq \aleph_1}$ is proved to be \textit{equivalent} to the continuum hypothesis (Corollary 2.16). 
 
 Further examples of algebraically complete and cocomplete categories, assuming GCH, are $K\mbox{-}\Vec_{\leq \lambda}$ for regular infinite cardinals
 $\lambda > |K|$ and
$$
\operatorname{\bf Nom}_{\leq \lambda}
$$
 the category of nominal sets of cardinality at most $\lambda $, for regular cardinals $\lambda > \aleph_1$.  And if $G$ is a group with $2^{|G|} <\lambda$, then the category
$$
G\mbox{-}\Set_{\leq \lambda}
$$
of $G$-sets (sets with an action of $G$)  of cardinality at most $\lambda$ is algebraically 
 complete and cocomplete.

If we work in the category of $cpo$'s as our base category,  then Smyth and Plotkin proved in \cite{SP}
 that the initial algebra coincides with the terminal coalgebra for all locally continuous endofunctors. That is, the underlying objects are equal, and the structure maps are inverse to each other.
Is there a connection between initial algebras $\mu F$ and terminal coalgebras $\nu F$ for set functors $F$, too?
We concentrate on \textit{precontinuous} set functors which is a generalization of finitary set functors encompassing also all continuous functors (preserving limits of $\omega^{\op}$-chains) and closed under arbitrary
 products, subfunctors, coproducts and composition. Each precontinuous functor has a terminal coalgebra $\nu F$ which carries a canonical partial order, and an initial algebra that, as a subposet of 
$\nu F$, has the same ideal \textit{cpo}-completion whenever $F\emptyset \ne \emptyset$. Consquently, assuming GCH, if the initial algebra is uncountable, it has the same cardinality as the terminal coalgebra.

 And, analogously, assuming GCH the initial algebra and  terminal coalgebra of a precontinuous functor with $F\emptyset \ne \emptyset$ carry a canonical ultrametric such that the Cauchy completions of $\mu F$ and $\nu F$ coincide. This complements the result of Barr \cite{B} that for every finitary, continuous set functor with $F\emptyset \ne \emptyset$ the metric space $\nu F$ is the Cauchy completion of $\mu F$. 

Finitary functors are also a topic of the last section devoted  to non-regular cardinals: if $\lambda$ is a cardinal of countable cofinality, we prove that every endofunctor of $\Set_{\leq \lambda}$ is  finitary, assuming that  no measurable cardinal
is smaller or equal to $\lambda$. For example for $\lambda = \aleph_\omega$. Surprisingly, $\Set_{\leq\aleph_\omega}$ is not algebraically complete.

\vskip 2mm
\noindent
\textbf{Related Work}  This paper extends results of the paper \cite{A1} presented at the conference CALCO 2019. The result about the ideal $cpo$-completion of the initial algebra and the last section on non-regular cardinals are new. In \cite{A1} a proof was presented that assuming GCH, a set functor $F$ having $\mu F$ of uncountable  regular  cardinality has $\nu F$ of the same cardinality. Unfortunately, the proof was not correct. We present in Corollary~\ref{C:fin} a different proof assuming $F$ is precontinuous.
 \vskip2mm
 \noindent
 \textbf{Acknowledgement} The author is grateful to the referees for a number of very useful comments. And to George Janelidze, Stefan Milius, Larry Moss and Walter Tholen  for valuable discussions.
 
 \section{Algebraically Cocomplete Categories}\label{sec2}
 
 For a number of categories $\ck$ we prove that the full subcategory $\ck_{\leq \lambda}$ on objects of power at most $\lambda$ is algebraically cocomplete. Power is  a cardinal defined below. 
 
 
 \begin{defi}\label{D:power}
 An object $A$ is called \textit{connected} if whenever it is isomorphic to a coproduct, then it is isomorphic to one of the summands.
 (In particular, A is not initial.)
 An object has \textit{power $\lambda$} if it is a coproduct of $\lambda$ connected objects and $\lambda$ is
 the least cardinal possible.
 \end{defi}

 \begin{exa}\label{E:power}
 In $\Set $, connected objects are the singleton sets, and power has its usual meaning. In the category $K$-$\Vec $ of vector spaces over a field $K$ the connected spaces are those of dimension one, and power means dimension. In the category $\Set^S$ of many-sorted sets the connected objects are those with precisely one element (in all sorts together), and the power of $X= (X_s)_{s\in S}$ is simply $|\coprod\limits_{s\in S} X_s|$.
 \end{exa}
 
 \begin{defi}\label{D:width}
A category $\ck$ is said to have \textit{width} $w (\ck)$ if it has  coproducts, every object is a coproduct of connected objects, and $w(\ck)$ is the smallest cardinal $\beta$ such that
 \begin{enumerate}
\item[(a)] $\ck$ has at most  $\beta$ connected objects up to isomorphism, and
\item[(b)] for all cardinals  $\alpha \geq  \beta$ there exist at most $\alpha$ morphisms from a connected object to an object of power $\alpha$.
\end{enumerate}
\end{defi}
 
  \begin{exa}\label{E:width}
  \hfill
  
  (1) $\Set$ has width $1$. More generally, $\Set^S$ has width $|S|$. Indeed, in Example~\ref{E:power} we have seen that the number of connected objects up to isomorphism is $|S|$, and (b) clearly holds.
  
  (2) $K$-$\Vec$ has width $|K| +\aleph_0$. Indeed, the  only connected object, up to isomorphism, is $K$. For a space $X$ of dimension $\alpha$ the number of  morphisms from $K$ to $X$ is $|X|$. If $K$ is infinite, then $\alpha \geq |K|$ implies $|X|= \alpha$ (and $|K|=|K|+\aleph_0$). For $K$ finite, the least cardinal $\lambda$ such that $|X| \leq \alpha$ holds for every $\alpha$-dimensional spaces $X$ with $\alpha \geq \lambda$ is $ \aleph_0$ ($= |K|+ \aleph_0$).
  
  (3) $G$-$\Set$, the category of sets with an action of the group $G$, has width at most $2^{|G|}$ for infinite groups, and at most $\aleph_0$ for finite ones, see our next lemma.  Recall that objects are pairs $(X, \cdot)$ where $X$ is a set and $\cdot$ is a function from $G \times X$ to $X$ (notation: $gx$ for $g\in G$ and $x\in X$) such that 
\begin{align*}
h(gx) &= (hg)x && \mbox{for\quad  $h,g\in G$ and  $x\in X$}\,,\\
\intertext{and}
ex&=x  && \mbox{for\quad  $x\in X$ ($e$ neutral in $G$)}\,.
\end{align*}
Morphisms are the \textit{equivariant} functions: those preserving the unary operation $g.-$ for every $g \in G$.
   
(4) The category $\operatorname{\bf Nom}$ of nominal sets has width  $\aleph_0$, see \ref{C:nom} below. Recall that for a given countably infinite  set  $\ba$, the group $S_f (\ba)$ of finite permutations consists of all composites of  transpositions. A \textit{nominal set} is a set  $X$ together with an action of the group $S_f(\ba)$ on it (notation: $\pi x$ for $\pi \in S_f(\ba)$ and $x \in X$) such that every element $x\in X$ has a  finite support. This means a finite subset $A\subseteq \ba$ such that for every finite permutation we have:
$$
\pi(a) =a \quad \mbox{for all } \ a\in A \ \mbox{implies}\  \pi x =x.\
$$
Morphisms are the equivariant  functions.

\end{exa}

  \begin{lem}\label{L:width}
  For every group $G$ the category $G$-$\Set$ of sets with an action of $G$ has width
  $$
  w(G\mbox{-}\Set) \leq 2^{|G|} + \aleph_0\,.
  $$
  \end{lem}
  
  \begin{proof}\hfill (1) We first observe
an important example of a $G$-set given by any equivalence relation $\sim$ on $G$ which is \textit{equivariant}, i.e., fulfils
$$
g\sim g' \Rightarrow hg \sim hg' \quad \mbox{for all\quad $g, g', h\in G$}\,.
$$
Then the quotient set $G/\sim$ is a $G$-set (of equivalence classes $[g]$)  w.r.t. the action $g[h] = [gh]$.  This $G$-set  is clearly connected.

\vskip 1mm
(2) Let $(X, \cdot)$ be a $G$-set. For every element $x\in X$ we obtain a subobject of $(X, \cdot)$ on the set
$$
Gx =\{ gx; g\in G\} \qquad \mbox{(the orbit of\ \
$x$)}\,.
$$
The equivalence on $G$ given by
$$
g\sim g' \quad \mbox{iff} \quad gx=g'x
$$
is equivariant, and the $G$-sets $Gx$ and $G/\!\sim$  are isomorphic. Moreover, two orbits are disjoint or equal: given $gx =hy$, then $x=(g^{-1}h)y$, thus, $Gx = Gy$.

\vskip 1mm
(3)
Every object $(X, \cdot)$ is a coproduct of  at most $|X|$ connected objects.
Indeed, let $X_0$ be 
  a choice class of the equivalence $\equiv$ given by
  $$
  x \equiv y \Leftrightarrow  Gx=Gy,
  $$ 
 then $(X, \cdot)$ is a coproduct of the orbits of $x$ for $x\in X_0$.

(4)
The number of connected objects, up to isomorphism, is at most $2^{|G|} + \aleph_0$. Indeed, it follows from the above that the connected objects are represented by precisely all $G/\!\sim$ where $\sim$ is an equivariant equivalence relation. If $|G| =\beta$ then we have at most  $\beta^\beta$ equivalence relations. For $\beta$ infinite, this is equal to $2^{|G|}$, for $\beta$ finite, this is smaller  than 
$2^{|G|} + \aleph_0$.

\vskip 1mm
(5)
The number of morphisms from $G/\!\sim$ to  an object $(X, \cdot)$ is at most $|X|$. Let $\alpha \geq 2^{|G|} + \aleph_0$. If the power of  $(X, \cdot)$ is $\alpha$, than the cardinality of $X$ is at most $\alpha$ since the cardinality of $X_0$ in (3) is at most $\alpha$. Consequently, there exist at most $\alpha$ morphisms $p: G/\!\sim \to (X,.)$ .
Indeed, every morphism $p$ is determined  by the value $x_0 =p([e])$ since $p([g]) = p(g[e]) = g \cdot x_0$  holds for all $[g]\in G/\!\sim$.
\end{proof}

\begin{rem}[Due to George Janelidze, a personal communication] The above does not generalize from groups to monoids.
For example, if $M$ is a commutative monoid, then the category of sets with the action of $M$ does not have a width unless $M$ is a group.
Indeed, suppose $M$ is not a group, then there exist arbitrarily large connected $M$-sets: consider a set X with an element x
and define the action $a: M \times X \to X$ by $a(m,-)=id$ if $m$ is invertible, else the constant map of value $x$.
\end{rem}

\begin{cor}\label{C:nom}
The category $\operatorname{\bf Nom}$ of nominal sets has width ${\aleph_0}$. 
\end{cor}

The proof is completely analogous to that of 2.5: in (2) each orbit $S_f(\ba)\big/\!\!\!\sim\ \simeq S_f(\ba)x$ is a nominal set. And the   number of all such orbits up to isomorphism is $\aleph_0$, see Lemma~A1 in \cite{AMSW}. In (5) we have $|X| \leq\alpha \cdot\aleph_0=\alpha$ for all $\alpha\geq \aleph_0$.

\vskip 3mm
The following lemma is based on ideas of V.~Trnkov{\'{a}}  \cite{T}:

\begin{lem} \label{L:tr}
A commutative square
$$
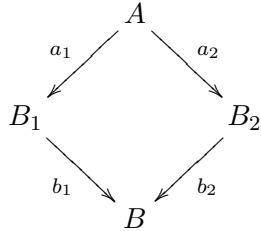
\xymatrix{  
& A \ar [dl]_{a_1} \ar[dr]^{a_2} &\\
B_1 \ar [dr]_{b_1} && B_2 \ar[dl]^{b_2}\\
& B &
}
$$
in  $\ca$ is an absolute pullback, i.e., a pullback preserved by all  functors with  domain $\ca$, provided that 

{\rm (1)} $b_1$ and $b_2$ are split monomorphisms, and

{\rm (2)} there exist morphisms $\bar b_1$ and $\bar a_2$:

$$
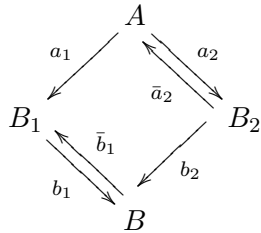
\xymatrix{  
& A \ar [dl]_{a_1} \ar[dr]^{a_2} &\\
B_1 \ar [dr]_{b_1} && B_2 \ar@<-2ex>[dl]^{b_2}
\ar@<1ex>[ul]^{\bar a_2}\\
& B \ar @<-1ex>[ul]_{\bar b_1} &
}
$$
satisfying
\begin{equation}\label{eq:tr}
\bar b_1 b_1=\id\,, \quad \bar a_2 a_2 =\id,\quad  \mbox{and}\quad  a_1 \bar a_2 =\bar b_1 b_2
\end{equation}
\end{lem}

\begin{proof}
The first square above is a pullback since given  a commutative square
$$
b_1c_1 = b_2 c_2 \quad \mbox{for}\quad c_i \colon C \to B_i
$$
there exists a unique $c$ with $c_i = a_i \cdot c$ ($i=1,2$). Uniqueness is clear since $a_2$ is split monic. Put $c= \bar a_2 \cdot c_2$. Then $c_1= a_1 c$ follows from $b_1$ being monic: 
\begin{align*}
b_1c_1&= b_1 \bar b_1 b_1 c_1 & \bar b_1 b_1 &=\id\\
&= b_1 \bar b_1 b_2 c_2 & b_1 c_1&= b_2 c_2\\
&= b_1  a_1\bar a_2c_2 & \bar b_1 b_2 &= a_1 \bar a_2\\
&= b_1 a_1 c & c&= \bar a_2 c_2
\end{align*}
And $c_2 =a_2 c$ follows from $b_2$ being monic:
\begin{align*}
b_2 c_2& = b_1 c_1 &&\\
&= b_1 a_1 c & c_1 &= a_1 c\\
&= b_2 a_2 c & b_1 a_1 &= b_2 a_2
\end{align*}

For every functor $F$ with domain $\ca$ the image of the given square satisfies the analogous conditions: $Fb_1$ and $Fb_2$ are split monomorphisms and $F\bar b_1$, $F\bar a_2$ verify \eqref{eq:tr}. Thus, that image is  a pullback, too.
\end{proof}

\begin{corC}[\cite{T}]\label{C:tr}
Every set functor preserves nonempty finite intersections.
\end{corC}

Indeed, if $A$ above is nonempty, choose an element $t\in A$ and define $\bar b_1$ and $\bar a_2$ by
\begin{align*}
\bar b_1 (x)&= \begin{cases}
 y & \mbox{if\ \ $b_1 (y) =x$}\\ 
a_1(t) & \mbox{if\ \  $x\notin b_1 [B_1]$}
 \end{cases}
 \\
 \intertext{and}
 \bar a_2(x) &= 
\begin{cases} y &\mbox{if\ \ $a_2(y) =x$}\\ 
t&\mbox{if\ \  $x\notin a_2 [A]$}
 \end{cases}
 \end{align*}
It is easy to see that \eqref{eq:tr} holds.

\begin{rem}\label{R:disj} 
%
%

(a)
	 We recall that 
the \textit{cofinality} of an infinite cardinal $\lambda$ is the smallest cardinal $\mu$ such that $\lambda$ is a join of a $\mu$-chain of smaller cardinals. An infinite cardinal is called \textit{regular} if it equals its cofinality. The first non-regular cardinal is $\aleph_\omega$.

(b)
	Recall that for a set $X$ of infinite  cardinality $\lambda$ a collection of subsets is called \textit{almost disjoint} if the intersection of arbitrary two distinct members  has cardinality smaller than $\lambda$.
	
	Tarski \cite{Ta} proved that for every set $X$ of  infinite cardinality $\lambda$ (not necessarily regular) there exists an almost disjoint collection $Y_i \subseteq X$ ($i\in I$) with 
	$$
	|I| >\lambda.
	$$
Moreover, we can assume $|Y_i| = \lambda$ for all $i$: see \cite{Baum}, Thm. 2.8.


  (c) 
  Given an element $t\in X$ there exists an almost disjoint collection $Y_i$ as above with $t\in Y_i$ for all $i\in I$. Indeed, take any almost disjoint collection $(Y_i)_{i\in I}$ and use  $Y_i \cup \{t\}$ instead (for all $i  \in I$).

  \end{rem}

  \begin{defi}\label{D:lambda} 
  Let $\ck$ be a category  of width $w (\ck)$.
    For every infinite cardinal $\lambda > w(\ck)$ we denote by
  $$
  \ck_{\leq \lambda}
  $$
  the full subcategory of $\ck$ on objects of power  at most $\lambda$.
    \end{defi}

  \begin{prop}\label{P:lambda} 
  Let $F$ be an endofunctor of $\ck_{\leq\lambda}$, $\lambda$ regular.
Given an object $X=\coprod\limits_{i\in I} X_i$, with all $X_i$ connected, every morphism       
  $b\colon B \to FX$  with $B$
 of power less than $\lambda$
factorizes through $F c$ for some subcoproduct
$c\colon C \to X$ where $C= \coprod\limits_{j\in J} X_j$
 and $|J|<\lambda$:
  $$
\xymatrix{  
& FC\ar[d]^{Fc}\\
B\ar@{-->}[ur] \ar [r]_b & FX
}
$$
\end{prop}

\begin{proof}\hfill

 (1)
It is sufficient to prove this in case $X$ has power precisely $\lambda$ (otherwise put $c=\id_X$). And we can also  assume that $B$ is connected. In  the general case we have, by Definition 2.1, a coproduct of connected objects $B=\coprod\limits_{k\in K} B_k$ with $|K| <\lambda$, and we find for each $k$ a coproduct injection  $c_k \colon C_k \to X$ corresponding to the $k$-th component of $b$, then put $C=\coprod\limits_{k\in K} C_k$ (which has power less than $\lambda^2 =\lambda$  since each $C_k$ does  and $|K|<\lambda$) and  put $c=[c_k] \colon C\to X$.

 Since $\lambda>w(K)$, there are less than $\lambda$ connected objects up to isomorphism. Thus, as $\lambda$ is regular, in the coproduct of $\lambda$ connected  objects representing $X$, at least one  component
 must appear $\lambda$ times. Thus $X$ has the form
 $$
 X = \coprod_\lambda R + X_0
 $$
 where both $R$ and $X_0$ have power less than $\lambda \colon X_0$ is the coproduct of all components that appear less than $\lambda$ times as $X_i$ for some $i\in I$ and $R$ is the coproduct of all the other components where isomorphic copies are not repeated.

 (2) By Remark \ref{R:disj} we can choose $t\in \lambda$ and an almost disjoint collection of sets $M_k \subseteq \lambda$, $k\in K$,  with
 $$
 t\in M_k\,, \ \ |M_k|=\lambda \quad \mbox{and} \quad |K|>\lambda\,.
 $$ 
 Denote for every $M\subseteq \lambda$ by $Y_M$ the coproduct
 $$
 Y_M = \coprod_{M} R+X_0
 $$
 and for every $k\in K$ let $b_k \colon Y_{M_k} \to X$ be the coproduct injection.

 Consider the following square of coproduct injections for any pair $k$, $l\in K$:
 $$
\xymatrix{  
& Y_{M_l\cap M_k} \ar[dl]_{a_k} \ar[dr]^{a_l}
&\\
Y_{M_k} \ar[dr]_{b_k} &&  Y_{M_l}\ar[dl]^{b_l}\\
& X & \\
}
$$
This is an absolute pullback. Indeed, it obviously commutes. And $b_k$ and $b_l$ are split monomorphisms: define
$$
\bar b_k \colon X \to Y_{M_k}
$$
as identity on the summand $ X_0$ whereas the $i$-th copy of $R$ is sent to copy $i$ if $i\in M_k$, and to copy $t$ else. Then
$$
\bar b_k b_k =\id\,.
$$
Analogously for $b_l$. Next define
$$
\bar a_l \colon Y_{M_l} \to Y_{M_k\cap M_l}
$$
as identity on the summand $ X_0$ whereas the $i$-th copy of $R$ is sent to copy $i$ if $i\in M_k$, and to copy $t$ else.
Then clearly
$$
\bar a_k a_l =\id \quad \mbox{and} \quad a_k\bar a_l = \bar b_k b_l\,.
$$
Thus, the above square is an absolute pullback by Lemma~\ref{L:tr}.
\vskip 2mm

(3) We are ready to prove that for a connected object $B$ every morphism
$$
b\colon B \to FX
$$
has the required factorization.
For every $k\in K$ since $|M_k|=\lambda$ we have an isomorphism
$$
y_k \colon X \to Y_{M_k}
$$
which composed with $b_k\colon Y_{M_k} \to X$ yields an endomorphism
$$
z_k = b_k\cdot y_k \colon X \to X\,.
$$
 The following morphisms
$$
B \xrightarrow{\ b\ } FX  \xrightarrow{\ Fz_k\ } FX \quad (k\in K)
$$
 are not pairwise distinct because  $|K| > \lambda$,
whereas 
$|\ck (B, FX)| \leq \lambda$. Indeed, the latter follows since
$FX$ has at most $\lambda$ components (since $F$ is an endofunctor of $\ck_{\leq \lambda}$) so that (b) in Definition~\ref{D:width} implies that $\ck (B, FX)$ has cardinality at most $\lambda$. Choose $k\ne l$ in $K$ with
\begin{equation}\label{eq:k-l}
F z_{k} \cdot  b = Fz_l \cdot  b\,.
\end{equation}
Compare the pullbacks $Z$ of $z_k$ and $z_l$ and $Y_{M_k\cap M_l}$ of $b_k$ and $b_l$:
$$
\xymatrix@R=3pc{
& Z \ar@{-->}[d]^{p} \ar[dl]_{p_k} \ar[dr]^{p_l} &\\
X \ar[d]_{y_k} & Y_{M_k\cap M_l}\ar[dl]_{a_k} \ar[dr]^{a_l}& X \ar[d]^{y_l}\\
Y_{M_k} \ar[dr]_{b_k} && Y_{M_l} \ar[dl]^{b_l}\\
& X & 
}
$$
Since $y_k$ and $y_l$ are isomorphisms, the connecting morphism $p$ between the above pullbacks is an isomorphism, too.  We know that $|M_k\cap M_l|<\lambda$ 
since $M_k$, $M_l$ are members 
 of our almost disjoint family, thus the coproduct injection
 $$
 c\colon Y_{M_k\cap M_l} \to X
 $$
 has  less than $\lambda$ summands, as required. And, due to (2),  the pullback of $z_k$ and $z_l$ is absolute. The equality \eqref{eq:k-l} thus implies that $b$ factorizes through $Fp_k$.
 Since clearly $p_k = c\cdot p$, this implies that $b$ factorizes through $Fc$, as required.
  \end{proof}
 
 \begin{rem}\label{R:nova} The assumption that $\lambda$ is a regular cardinal has only been used  in the above proof
 	in Step (1), where we claimed that one component is repeated $\lambda$ times in $X$. If our category has finitely many connected objects up to isomorphism, then the above proposition also holds for non-regular infinite cardinals.
 \end{rem}

 \begin{rem}\label{R:lambda}
 \hfill

%
%
 (a) If an infinite cardinal $\lambda$ has cofinality $n$, then $\lambda^n>\lambda$, see \cite{J}, Corollary 1.6.4.

 (b) Every ordinal $\alpha$ is considered as the set of all smaller ordinals. In particular $\aleph_0$ is the set of all natural numbers and $\aleph_1$ the set of all countable ordinals.

 (c) Recall the
  \begin{center}
 	\textit {Continuum Hypothesis (CH)}
 \end{center}
stating that the cardinal successor of $\aleph_0$ is 
$2^{\aleph_0}$, and the
 \begin{center}
 \textit{General Continuum Hypothesis (GCH)}
 \end{center}
 which states that for every infinite cardinal $\lambda$ the cardinal successor is $2^\lambda$. 

(d) Under GCH every infinite regular cardinal $\lambda$ fulfils
$$
\lambda^n =\lambda \quad \mbox{for all cardinals} \quad 1\leq n<\lambda\,.
$$
 See Theorem 1.6.17 in \cite{J}.
 \end{rem}
 
 \begin{thm} \label{T:lambda}
 Assuming GCH, let $\ck$ be a cocomplete and cowellpowered category of  width $w(\ck)$. Then $\ck_{\leq\lambda}$ is algebraically cocomplete for all uncountable regular cardinals $\lambda > w(\ck)$.
 \end{thm}

 \begin{proof}
 Let $F$ be an endofunctor of $\ck_{\leq \lambda}$. Form a collection $a_i \colon A_i \to FA_i$ ($i\in I$) representing all coalgebras of $F$ on objects of power  less than $\lambda$  (up to isomorphism of  coalgebras).
 We have
 $$
 |I|\leq \lambda\,.
 $$
Indeed, for every cardinal $n<\lambda$ let $I_n\subseteq I$ be the  subset of all $i$ with $A_i$ having $n$ components. Given $i\in I_n$, for every component $b\colon B \to A_i$ of $A_i$ we know, since  $\lambda > |\ck|$, that there are at most $\lambda$ morphisms from $B$ to $FA_i$ (recalling that $FA_i$ has at most $\lambda$ components).  Thus there are at most $n\cdot \lambda =\lambda$ morphisms from $A_i$ to $FA_i$. And  the number of objects $A_i$ with $n$ components is at most $w (\ck)^n< \lambda^n =\lambda$ (see Remark \ref{R:lambda}). Thus, there are at most $\lambda$ indexes in $I_n$. Since $I=\bigcup\limits_{n<\lambda} I_n$, this proves $|I| \leq \lambda^2 =\lambda$.

 Consequently $A=\coprod\limits_{i\in I} A_i$ is an object of $\ck_{\leq \lambda}$, and we have the coalgebra structure $\alpha\colon A\to FA$ of a coproduct of $(A_i, \alpha_i)$ in  $\Coalg F$. 
Let $e\colon (A, \alpha) \to (T, \tau)$ be  the wide pushout  of all homomorphisms in $\Coalg F$ with domains $(A, \alpha)$ carried by epimorphisms of $\ck$. Since $\ck$ is cocomplete and cowellpowered, and  since the forgetful functor from $\Coalg F$ to $\ck$ creates colimits, this means that we form the corresponding pushout in $\ck$ and get a unique coalgebra structure
 $$
 \tau \colon T \to FT
 $$
 making $e$ a homomorphism:
 $$
\xymatrix{  
\coprod A_i \ar[r]^{\alpha} \ar[d]_{e}  & F(\coprod A_i) \ar[d]^{Fe}\\
T \ar[r]_{\tau} & FT
}
$$
  We are going to prove that $(T, \tau)$ is a terminal coalgebra.

 (1) Firstly, for every coalgebra $\beta \colon B \to FB$ with $B$ having power  less than $\lambda$ there exists a unique homomorphism into $(T, \tau)$. Indeed, the existence is clear: compose  the isomorphism that exists from $(B, \beta)$ to some $(A_i, \alpha_i)$, the $i$-th coproduct injection to $(A, \alpha)$ and the above  homomorphism $e$. To prove uniqueness, observe that by definition of $(T, \tau)$, this coalgebra has no nontrivial quotient: every homomorphism with domain $(T, \tau)$ whose  underlying morphism is epic in $\ck$ is invertible. Given homomorphisms $u$, $v \colon (B, \beta) \to (T, \tau)$
 $$
\xymatrix{  
B\ar[r]^\beta
\ar@<.5ex>[d]^v \ar@<-.5ex>[d]_u
 & FB \ar@<.5ex>[d]_{Fu\ } \ar@<-.5ex>[d]^{\ Fv}\\
 T\ar[r]^\tau \ar[d]_{q} & FT\ar[d]^{Fq}\\
 Q \ar@{-->}[r] & FQ
}
$$
form their coequalizer $q\colon T\to Q$ in $\ck$. Then $Q$ carries the structure of a coalgebra making $q$ a homomorphism. Thus, $q$ is invertible, proving $u=v$.

(2) Next, consider an arbitrary coalgebra $\beta \colon B \to FB$. Express $B =\coprod\limits_{i\in I} B_i$ where $B_i$ are connected and  assume  $|I|=\lambda$ (the case $|I|<\lambda$ has just been handled). For every set $J\subseteq I$ we
denote by $u_J \colon \coprod\limits_{i\in J} B_i\to B$ the subcoproduct. In case
 $|J|<\lambda$ we are going to prove that there exists a set $J\subseteq J' \subseteq I$ with $|J'| <\lambda$  such that the summand 
$$ 
u_{J'} \colon B_{J'} =\coprod_{i\in J'} B_i \to B
$$
carries a subcoalgebra. That is, there exists $\beta_{J'} \colon B_{J'} \to FB_{J'}$ for which  $u_{J'}$ is a homomorphism. Indeed, we put
$$
J' =\bigcup_{n<\omega} J_n
$$
for the following $\omega$-chain of sets $J_n \subseteq I$ with $|J_n| <\lambda$. First
$$ 
J_0=J\,,
$$
and given $J_n$, apply Proposition~\ref{P:lambda} to $\beta . u_{J_n} \colon B_{J_n} \to FB$. We conclude that this morphism factorizes through $Fb_{J_{n+1}}$ for some subset $J_{n+1}\subseteq I$ of power less that $\lambda$:
$$ 
\xymatrix{
B_{J_n} \ar @{-->}[r]^{\beta^n} 
\ar [d]_{u_{J_n}} & FB_{J_{n+1}} \ar [d]^{Fu_{J_{n+1}}}\\ 
B \ar[r]_{\beta} & FB
}
$$

Thus, for the union $J'=\bigcup J_n$ we get  $|J'|<\lambda$ because $\lambda$ is  uncountable and regular, therefore $|\underset{n<\omega}{\coprod}J_n| <\lambda$.
And $u_{J'}$ carries the following subcoalgebra $\beta_{J'}\colon \coprod\limits_{j\in J'}B_j \to F \big(\coprod\limits_{j\in J'} B_j\big)$ of $(B, \beta)$:
Given $j\in J'$ let $n$ be the least number with $j\in J_n$. Denote by 
$w\colon B_j \to \coprod\limits_{i\in J_n} B_i$ and 
 $v\colon \underset{i \in J_{n+1}}{\coprod} B_i \to \underset{j \in J'}{\coprod} B_j$ the coproduct injections. Then the $j$-th component of $\beta'$ is the following composite
$$
B_j \xrightarrow{w} \coprod_{i\in J_n} B_i \xrightarrow{\beta^n} F\big( \underset{i\in J_{n+1}}{\coprod} B_i\big) \xrightarrow{Fv} F\big(\coprod_{i\in J'} B\big)
$$
 To prove that the  square below
$$
\xymatrix@C=3pc@R=3pc{  
\coprod\limits_{i\in J'} B_i \ar[r]^{\beta_{J'}\ \ } \ar[d]_{u_{J'}} & F\big(\coprod\limits_{i\in J'} B_i\big) \ar [d]^{Fu_{J'}}\\
\coprod\limits_{i\in I} B_i \ar[r]_{\beta\ \ } & F\big(\underset{i\in I}{\coprod} B_i\big)
}
$$ 
 commutes, consider  the components for $j\in J'$ separately. The upper passage yields, since
 $u_{J'} \cdot v =  u_{J_{n+1}} \colon \underset{i\in L_n}{\coprod}  B_i \to \underset{i\in I}{\coprod} B_i$, the result
 $$
 Fu_{J'} \cdot (Fv\cdot \beta^n\cdot w) = Fu_{J_{n+1}} \cdot \beta^n \cdot w = \beta \cdot u_{J_n} \cdot w\,.
 $$
 The lower passage yields the same.
 
 (3) We conclude that every coalgebra $\beta \colon B \to FB$ is a colimit, in $\Coalg F$, of coalgebras on objects of power less than $\lambda$. Indeed,  this is non-trivial only for objects $B\in \ck_{\leq \lambda}$ where 
 $B=\underset{i\in I}{\coprod} B_i$ with $B_i$ indecomposable and $|I| =\lambda$.  Then for every set $J'\subseteq I$ with $|J'|<\lambda$ for which  $B_{J'} $ carries a subcoalgebra $(B_{J'}, \beta_{J'})$ of $(B, \beta)$ we take this as an object of our diagram. And morphisms from 
 $(B_{J'}, \beta_{J'})$ to $(B_{J^{\prime\prime}}, \beta_{J^{\prime\prime}})$ (where again $J^{\prime\prime} \subseteq I$ fulfils $|J^{\prime\prime}| <\lambda)$, are those   coproduct injections $u\colon B_{J'} \to B_{J^{\prime\prime}}$ for $J'\subseteq J^{\prime\prime}$ that are coalgebra homomorphisms. Since every subset $J\subseteq I$ with $|J|<\lambda$ is contained in $J'$ for some object of this diagram, it easily follows that the  cocone of coproduct injections from $(B_{J'}, \beta_{J'})$ to $(B, \beta)$ is a colimit of the above diagram.
 
(4)  Thus, from Part (1) we conclude that there exists a unique homomorphism from $(B, \beta)$ to $(T, \tau)$.
   \end{proof}
 
\begin{exa}\label{E:new} 
The category $\Set_{\leq \lambda}$ of sets of power at most $\lambda$ is, under GCH, algebraically cocomplete for all uncountable regular cardinals $\lambda$.

However, the category $\Set_{\leq \aleph_0}$ of countable sets is not algebraically cocomplete.
 Indeed, the restriction $\cp_f$ 
of the finite power-set functor to it does not have a terminal coalgebra.

Assuming the contrary, let $\tau \colon T\to \cp _f T$ be  a (countable) terminal coalgebra for the restricted functor. For every subset $A\subseteq \N$ denote by $X_A$ the following countable, finitely branching graph:
$$
\xymatrix{&&&&&&\\
0 \ar[r] & 1 \ar[r] & 2\ar[r] & \cdots\hskip-1cm&\hskip-1cm\ar[r]& i\ar[r] \ar[ur] & \cdots 
}$$
\vskip-0.7cm\hskip 11cm $(i\in A)$
\vskip .8cm
\noindent
It is obtained from the infinite path on $\N$ by taking an extra successor for $i\in \N$ iff $i\in A$. Thus $X_A$ is a coalgebra for $\cp_f$. The unique homomorphism $h_A\colon X_A \to (T, \tau)$ takes $0$ to an element $h_A(0)\in T$. The desired contradiction is achieved by proving for all subsets $A$, $A'$ of $\N$ that if $h_A(0) = h_{A'}(0)$,  then  $A=A'$ -- but  then $|T|\geq |2^{\N}|$.

From $h_A(0) = h_{A'}(0)$ we derive $A\subseteq A'$; by symmetry, $A=A'$ follows.
 Since $h_A$ and $h_{A'}$ are coalgebra homomorphisms, the relation $R\subseteq X_A \times X_{A'}$ of all pairs $(x, x')$ with $h_A(x) = h_{A'}(x')$ is a graph bisimulation. By assumption $0R0$, and  for every $i\in A$ we find a leaf of $A$ with distance $i+1$ from $0$. Hence, there is a leaf of $A'$ with  the same distance. This proves $A\subseteq A'$.
\end{exa}     
     
\begin{cor} The category $\Set_{\leq \aleph_1}$ is algebraically cocomplete iff the continuum hypothesis holds.
\end{cor}
     
Indeed, if CH holds, then the proof of Theorem 2.14 applies: we do not need the full strength of GCH in that proof. All we need is the statement at the beginning of the proof that for every $n<\aleph_1$ there are at most $\aleph_1$ coalgebras of power $n$, which follows from CH.

Conversely, if $\aleph_1 < 2^{\aleph_0}$, then the restriction of  $\cp_f$ to  $\Set_{\leq \aleph_1}$ does not have 
a terminal coalgebra, the argument is as in the preceding example. 
     \section{Algebraically Complete Categories} 
     
    In this section we prove a criterion for a category to be algebraically complete. It does not require GCH. But we will have to
    assume a bit more than the above width. Fortunately, all of our concrete examples above satisfy that stricter property.     

\begin{rem}\label{R:const}
\hfill

(a) Let $\ck$ be a category with an initial object $0$ and with colimits of $i$-chains for all limit ordinals $i\leq \lambda$.
 Recall from \cite{A0} the \textit{initial-algebra chain} of an endofunctor $F$: its objects  $F^i 0$ for all ordinals $i\leq \lambda$ and its connecting  morphisms $w_{i,j} \colon F^i 0 \to F^j 0$ for all $i\leq j\leq \lambda$ are defined by transfinite recursion as follows:
\begin{align*}
& F^0 0 =0\,,\\
& F^{i+1} 0 = F(F^i 0), \quad \mbox{and}\\
& F^j 0 = \underset{i<j}{\colim} F^i 0 \quad \mbox{for limit ordinals \ $j$.}
\end{align*}
Analogously:
\begin{align*}
& w_{01} \colon 0 \to F0 \quad \mbox{is unique},\\
& w_{i+1, j+1} = Fw_{i,j}\,,
\quad \mbox{and}\\
& w_{ij} \ (i<j) \quad \mbox{is a colimit cocone for limit ordinals $j$.}
\end{align*}

(b) The initial-algebra chain  \textit{converges} at a limit ordinal $\lambda$ if the connecting map $w_{\lambda, \lambda+1}$ is invertible. In that case we  get the initial-algebra

$$
\mu F = F^\lambda 0
$$
with the algebra structure
$$
\iota = w_{\lambda, \lambda+1}^{-1} \colon F( F^\lambda 0)\to F^\lambda 0\,.
$$
Following \cite{TAKR} such an ordinal exists whenever $F$ has a fixed point, in particular, whenever $F$ has a terminal coalgebra.\

(c) In particular, if  $F$ preserves colimits of $\lambda$-chains, then $\mu F = F^\lambda 0$.

\vskip 1mm
(d) Dually, the \textit{terminal-coalgebra chain} has objects $F^i1$  with $F^0 1=1$, $F^{i+1} 1= F(F^i 1$) and  $F^i 1 =\lim\limits_{i<j} F^i 0$ for limit ordinals $j$. Its connecting morphisms are denoted by $v_{ij}\colon F^i 1 \to F^j 1$\ ($i\geq j$).
For limit ordinals $j$ they are determined by the fact that $F^j 1$ is a limit of the preceding chain. For example 
$v_{\omega+ 1, \omega}$  is determined by the condition
 $$
v_{\omega+ 1, \omega} \cdot  v_{\omega, n+1} = v_{\omega +1, n+1} = Fv_{\omega,n}  
$$
for all $n<\omega$. 
   
   This was explicitly  formulated by Barr \cite{B}.

If the base  category has limits of $i^{\op}$-chains for all limit ordinals $i\leq \lambda$ and
 $F$ preserves limits of $\lambda^{\op}$-chains, then  $\nu F = F^\lambda 1$.
   
\vskip 1mm
 (e)  A set functor $G$ \textit{preserves inclusion}  if given a subset $X \subseteq Y$, then $FX \subseteq FY$ and for the inclusion map $i: X \to Y$ also $Fi$ is the inclusion map. 
 
 For every set functor $F$ there exists a set functor  $G$ preserving finite intersections and having the same initial-algebra chain as $F$ from $\omega$ onwards. Moreover, $F$ and $G$ coincide on all nonempty sets and functions and if $F\emptyset \ne \emptyset$, then $G\emptyset \ne \emptyset$. See \cite{AK}, Remark 3. Moreover $G$ can be chosen so that it preserves inclusion, see \cite{AT}, Theorem.III.4.5.


   \vskip 1mm
   (f) For every set functor $F$ there exists a unique morphism $\bar u\colon F^\omega 0 \to F^\omega 1$ making the squares below
   $$
\xymatrix{
F^n 0 \ar[r]^{w_{n,\omega}} \ar[d]_{F^n !} & F^\omega 0 \ar[d]^{\bar u} \\
F^n 1 & F^\omega 1 \ar[l]^{v_{\omega, n}}
} 
\xymatrix{&\\
\qquad (n\in {\mathbb{N}})&
}
$$
commutative (where $!\colon 0\to 1$ is the unique morphism). See \cite{A}, Lemma 2.4.
\end{rem}

 
 \begin{rem}\label{R:lp} 
We recall here some facts from \cite{AR}. Let $\lambda$ be an infinite regular cardinal.

(a) An object $A$ of a category $\ck$ is called  \textit{$\lambda$-presentable} if its hom-functor $\ck (A, -)$ preserves  $\lambda$-filtered colimits. And $A$ is called \textit{$\lambda$-generated} if $\ck (A,-)$ preserves $\lambda$-filtered colimits of diagrams where all connecting morphisms are monic.

(b) A category $\ck$ is called \textit{locally $\lambda$-presentable} if it is cocomplete and has a strong generator consisting of $\lambda$-presentable objects $G_i$ for $i \in I$. That is, every subobject $m:A \to B$ such that all morphisms $G_i \to B$
factorize through $m$ is invertible.

(c) Every such category has a factorization system (strong epi, mono). An object is $\lambda$-generated iff it is a strong quotient  of a $\lambda$-presentable one.

(d) In the case $\lambda =\aleph_0$ we speak about \textit{locally finitely presentable categories}.
\end{rem}

\begin{defiC}[\cite{AMSW}]\label{D:lp} 
A \textit{strictly  locally $\lambda$-presentable category} is a locally $\lambda$-presentable category in which every morphism $b\colon B\to A$ with  $B$ $\lambda$-presentable has a factorization
$$
b=b'\cdot f\cdot b
$$
for some morphisms $b'\colon B'\to A$ and $f\colon A\to B'$ with $B'$ also $\lambda$-presentable. 
\end{defiC}

\begin{exasC}[\cite{AMSW}]\label{E:lp} 
\hfill

(a) The categories $\Set$, $K$-$\Vec $ and $G$-$\Set$, where $G$  is a finite group, are strictly locally finitely presentable.

(b) $\Nom$ is strictly locally $\aleph_1$-presentable.

(c) $\Set^S$ is strictly locally $\lambda$-presentable iff $\lambda >|S|$.

(d) Given an infinite group $G$, the category $G$-$\Set$ is strictly locally $\lambda$-presentable if $\lambda >2^{|G|}$.
\end{exasC}

\begin{defi}\label{D:slp} 
A category $\ck$ has \textit{strict width} $w(\ck)$ if it has width $w(\ck)$, its coproduct injections are monic,  and every connected object is finitely presentable.
\end{defi}

\begin{exa}\label{E:slp} \hfill

(1) The category $\Set^S$ has strict width $|S|+ \aleph_0$, since connected objects (see Example~\ref{E:power}) are finitely presentable.

(2) $K$-$\Vec$ has strict width $|K| +\aleph_0$: the only connected object $K$ is finitely presentable.

(3) $G$-$\Set$ has strict width at most $2^{|G|} + \aleph_0$. Indeed, it follows from the proof of Lemma~\ref{L:width} that the only connected objects are the quotients of $G$, and they are easily seen to be finitely presentable.

(4) $\Nom$ has strict width $\aleph_1$, the argument is as in (3).
\end{exa}

\begin{lem}\label{L:slp} 
If a category has strict width  $w(\ck)$, then for every regular cardinal $\lambda \geq w(\ck)$  $\lambda$-presentable objects are precisely the objects of power less than $\lambda$.
\end{lem}

\begin{proof}
If $X$ is $\lambda$-presentable and $X =\coprod\limits_{i\in I} X_i$ with  connected objects $X_i$, then  in case  $\card I< \lambda$ we have nothing to prove.
 And if $\card I \geq \lambda$, form the $\lambda$-filtered diagram of all coproducts $\coprod\limits_{j\in J} X_j$ where $J$ ranges over  subsets of $I$ with $\card J<\lambda$. Its colimit is $X$. Since $\ck (X, -)$ preserves this colimit, there exists a factorization of $\id_X$ through one of the colimit injections $v\colon \coprod\limits_{j\in J} X_j \to  X$. Now  $v$ is monic (by the definition of strict width) and split epic, hence it is an isomorphism. Thus, $X \simeq \coprod\limits_{j\in J} X_j $ has power at most  $\card J<\lambda$.
 
 Conversely, if $X$ has power less than $\lambda$, then it is $\lambda$-presentable because every coproduct of less than $\lambda$ objects which are  $\lambda$-presentable   is $\lambda$-presentable.
 \end{proof}
 
 \begin{rem}\label{R:gener}
   \leavevmode

 (a) In a strictly locally $\lambda$-presentable category every $\lambda$-generated object is $\lambda$-presentable. This was  proved for  $\lambda= \aleph_0$ in \cite{AMSW}, Remark 3.8, the general case  is analogous.
 
 (b) In every locally $\lambda$-presentable category $\ck$ all $\hom$-functors of $\lambda$-presentable objects collectively reflect  $\lambda$-filtered colimits. That is, given a $\lambda$-filtered diagram $D$ with objects $D_i\ (i\in I)$, then a cocone 
 $c_i \colon D_i\to C$ of $D$ is a colimit iff for every $\lambda$-presentable object $Y$ the following holds:
 \begin{enumerate}[leftmargin=2cm,label={(\roman*)}]
 \item every morphism $f\colon Y \to C$ factorizes through some $c_i$
 \end{enumerate}
 and
 \begin{enumerate}[resume*]
 \item given two such  factorizations $u$, $v\colon Y\to C$, $c_i\cdot u = c_i\cdot v$, there exists  a connecting morphism $d_{ij} \colon D_i\to D_j$ of $D$ with $d_{ij} \cdot u = d_{ij}\cdot v$.
 \end{enumerate}
This is stated as an exercise in \cite{AR}, and proved for  $\lambda =\aleph_0$  in \cite{AMSW}, Lemma~2.7.

(c) Let $\cd$ be a full subcategory of $\ck$ representing all $\lambda$-presentable objects up to isomorphism. It follows from (b) that every object $K \in \ck$ is a colimit of the $\lambda$-filtered diagram of all morphisms $a \colon A \to K$ with $ A \in \cd$. More precisely, colimit of the forgetful functor  of the slice category, $D \colon \cd/K \to \ck$ with the canonical cocone given by all $a$'s.
\end{rem}

\begin{thm}\label{T:pres}
Let $\alpha$ be  a regular cardinal  such that $\ck$ is a strictly locally $\alpha$-presentable category with a strict width.
Then  $\ck_{\leq \lambda}$ is algebraically complete for every cardinal $\lambda \geq \max (\alpha, w(\ck))$.
\end{thm}

\begin{proof}
Following Remark 
\ref{R:const}, it is sufficient to prove that 
 $\ck_{\leq\lambda}$ has colimits of $i$-chains for all ordinals $i\leq \lambda$, and
   every endofunctor of $\ck_{\leq \lambda}$ preserves colimits of $\lambda$-chains.

(1) We first prove that $\ck_{\leq \lambda}$ is closed in $\ck$ under strong quotients $e\colon B\to C$. Thus, assuming that $B$ has power at most $\lambda$, we prove  the same about $C$. Let $B$ be a coproduct of connected objects $B_i (i \in I)$ with $|I| \leq\lambda$.
For every $i\in I$ the component $e_i\colon B_i \to C$ of $e$ factorizes by Proposition~\ref{P:lambda} through the subcoproduct $\coprod\limits_{j\in J(i)} C_j$ for some $J(i) \subseteq J$ of power less than $\lambda$. The union $K=\bigcup\limits_{i\in I} J(i)$ has power at most $\lambda^2 =\lambda$, and $e$ factorizes through the coproduct injection $v_k\colon\coprod\limits_{j\in K} C_j \to C$.
 Since $e$ is a  strong epimorphism, so is $v_K$. But being a coproduct injection, $v_K$ is also monic. Thus $v_K$ is an isomorphism, proving that  $C =\coprod\limits_{j\in K} C_j$ has  power at most $ |K|\leq \lambda$.

\smallskip
(2) $\ck_{\leq \lambda}$ has for every limit ordinal $i\leq \lambda$ colimits of $i$-chains $(B_j)_{j<i}$. In fact, let $X$ be the colimit of that chain in $\ck$, then we verify that $X$ has power at most $\lambda$. Indeed, $X$ is a strong quotient of $\coprod\limits_{j<i} B_j$. Each $B_j$ is a coproduct of at most $\lambda$ connected objects, thus, $\coprod\limits_{j<i} B_j$ is a coproduct of at most $i\cdot \lambda =\lambda$ connected objects. Due to (1) the same holds for~$X$.

\smallskip
(3) For every endofunctor $F\colon \ck_{\leq \lambda} \to \ck_{\leq \lambda}$ and every $\lambda$-chain $B_i$ ($i<\lambda$) in $\ck_{\leq}$ we prove that $F$ preserves the colimit
$$
X =\underset{i <\lambda}{\colim}\, B_i \quad \mbox{(with cocone $b_i\colon B_i\to X$, $i<\lambda $)}.
$$

\smallskip
Let us choose a small full subcategory $\cd$ of $\ck$ representing all $\alpha$-presentable objects. By (c) in the previous remark
 $X$ is a canonical colimit of the $\alpha$-filtered diagram 
$$
D\colon \cd\big/X \to \ck_{\leq \lambda}\,, \quad D(A,a) =A\,.
$$
The functor $B\colon \lambda \to \cd\big/X$ given by $i\mapsto (B_i, b_i)$ is final, i.e., for every object $(A,a)$ of $\cd\big/X$\ (a) there exists a morphism of $\cd\big/X$ into some $(B_i, b_i)$ and (b) given a pair of morphisms $u$, $v \colon (A,a) \to (B_i, b_i) $ there exists $j\geq i$ with $u$ and $v$ merged by the connecting morphism $b_{ij} \colon B_i \to B_j$ of our chain.
Indeed, since $A$ is $\alpha$-presentable and $\lambda \geq \alpha$, the morphism $a\colon A\to \underset{i<\lambda}{\colim}\, B_i $ factorizes through $b_i$ for some $i<\lambda$. And since $u$, $v$ above fulfil $b_i \cdot u=b_i \cdot v$ ($=a$), some connecting morphism $b_{ij}$ also merges $u$ and $v$.

Consequently, in order to prove that $F$ preserves the colimit $X=\colim B_i$, it is sufficient to verify that it preserves the colimit of  $ D_0$, where $D_0\colon \cd\big/X \to \ck_{\leq \lambda}$ is the codomain restriction of  $D$ above: since $B\colon \lambda \to \cd\big/X$ is final, the colimits of the diagrams $F\cdot D_0$ and $(FB_i)_{i<\lambda}$ coincide. We apply Remark~\ref{R:gener}(b) with $\alpha$ in place of $\lambda$, and verify the  conditions (i) and (ii) for the cocone $Fa \colon FA \to FX$ of $F \cdot D$ (in $\ck$). Thus $FX=\colim F \cdot D$ in $\ck$ which implies $FX=\colim F \cdot D_0$ in $\ck_{\leq\lambda}$.

\medskip
Ad (i) Given a morphism $f\colon Y\to FX$ with $Y$ $\alpha$-presentable, then $Y$ has by Lemma~\ref{L:slp} power less than $\alpha$, thus, by Proposition~2.9 there exists a coproduct injection $c\colon C\to X$ with $C$ $\alpha$-presentable such that $f$ factorizes through $Fc$ (which is a member of our cocone).

 \medskip
 Ad (ii) Let $u$, $v\colon Y\to FA$, with $A$ $\alpha$-presentable, fulfil $Fa \cdot u = Fa\cdot v$. We are to find a connecting morphism

 $$
 h\colon (A, a) \to (B, b)\quad \mbox{in\ \ $\cd\big/X$}
 $$
 with $Fh\cdot u = Fh \cdot v$. By the strictness of $\ck$, since $A$ is $\alpha$-presentable, for $a\colon A\to B$ there exist morphisms
 $$
 b\colon B\to X \quad \mbox{and} \quad f\colon X\to B
 $$
 with $B$ $\alpha$-presentable and $a=b\cdot f\cdot a$. It is sufficient to put
 $$ 
 h= f\cdot a\colon A\to B\,.
 $$
 Then $Fa \cdot u = Fb\cdot v$ implies $Fh\cdot u = Fh \cdot v$, as desired.
 \end{proof}
 
%

 \begin{exa}\label{E:set}
\hfill

(1) The category $\Set_{\leq \aleph_0}$ of countable sets is algebraically complete, but  not algebraically cocomplete (Example
\ref{E:new}).
 
 (2)  For every uncountable regular cardinal $\lambda$ the category $\Set_{\leq\lambda}$ is algebraically complete (by Theorem~\ref{T:pres}) and, assuming GCH, algebraically cocomplete (by Theorem~\ref{T:lambda}).
 The former was already proved in \cite{AK-79}, Example~14, using an entirely different method.

(3) For non-regular uncountable cardinals $\lambda$ the category $\Set_{\leq \lambda}$ need not be algebraically cocomplete, see Example 5.1.
 
\end{exa}

\begin{exa}\label{E:general} Let $\lambda$ be a regular uncountable cardinal.
The following categories are algebraically complete and, assuming GCH, algebraically cocomplete:

(a) $\Set_{\leq \lambda}^S$ whenever $\lambda > |S|$,

(b) $K$-$\Vec_{\leq \lambda}$  whenever $\lambda>|K|$,

(c) $\Nom_{\leq \lambda}$, and

(d) $G$-$\Set_{\leq \lambda}$ for groups $G$ with $\lambda > 2^{|G|}$.

\noindent
This follows from Theorems~\ref{T:lambda} and \ref{T:pres}.
\end{exa}


 \section{Precontinuous Set Functors and ultrametrics} 

 In the preceding parts we have presented  categories in which every endofunctor $F$ has an initial algebra $\mu F$ and a terminal coalgebra $\nu F$. The category of sets is, of course, not one of them. However, for set functors  we know that $\mu F$ exists whenever $\nu F$ does, see \cite{TAKR}. We observe below that $\mu F$, as a  coalgebra, is a subcoalgebra of $\nu F$. Moreover, 
 we introduce a wide class of set functors we call \textit{precontinuous} and for them we will see, assuming GCH, that $\nu F$ carries a canonical ultrametrics such that, whenever $F\emptyset \neq \emptyset, $ 
 \begin{enumerate}
 \item[(1)] the ultrametric subspace $\mu F$ has the same Cauchy completion as $\nu F$
 \end{enumerate}
 and
 \begin{enumerate} 
\item[(2)] the coalgebra  structure $\tau\colon \nu F \to F(\nu F)$ is  the unique continuous extension of $\iota^{-1}$, the inverted algebra structure $\iota \colon F(\mu F)  \to \mu F$.
\end{enumerate}
Thus one can say that the terminal  coalgebra  is determined, via its ultrametric, by the initial algebra.

For finitary set functors which also preserve limits of  $\omega^{op}$-sequence Barr proved more: $\nu F$ is a complete space which is the Cauchy completion of $\mu F$, see \cite{B}.

 \begin{prop}[$\mu F$ as a subcoalgebra of $\nu F$]\label{P:sub}
 If a set functor $F$ has  a terminal coalgebra, then it also has an initial algebra carried by a subset 
 $$
 \mu F \subseteq \nu F
 $$
 such that the inclusion map $m \colon \mu F \hookrightarrow \nu F$ is the unique coalgebra homomorphism
 $$
\xymatrix{ \mu F \ar[r]^{\iota^{-1}} \ar[d]_m & F(\mu F)
\ar[d]^{Fm}\\
\nu F \ar[r]_{\tau\ } & F(\nu F)
} 
$$
\end{prop}

\begin{proof}
(1) Assume first that $F$ preserves monomorphisms. Form the unique cocone of the 
initial-algebra chain (see Remark~\ref{R:const}) with co\-do\-main $\nu F$,
$$
m_i \colon F^i 0 \to \nu F \quad ( i\in \Ord)
$$ 
satisfying the recursive rule
$$
m_{i+1} \equiv F(F^i0) \xrightarrow{\ Fm_i\ } F (\nu F) \xrightarrow{\ \tau^{-1}\ } \nu F \quad (i\in \Ord)\,.
$$
Easy transfinite   induction verifies that $m_i$ is monic for every $i$. Since $\nu F$ has only a set of subobjects, there exists  an ordinal $\lambda$ such that all $m_i$ with  $i\geq \lambda$ represent the same subobject. Thus the commutative triangle below
$$
\xymatrix{ 
W_\lambda \ar[rr]^{w_{\lambda, \lambda+1}} \ar[dr]_{m_\lambda}
&& FW_\lambda \ar[dl]^{m_{\lambda+1}}\\
& \nu F &
} 
$$
implies that $w_{\lambda, \lambda+1}$ is invertible. Consequently, the following algebra
$$
F(F^\lambda 0) \xrightarrow{\ w_{\lambda, \lambda +1}^{-1}\ } F^\lambda 0
$$
is initial, see Remark~\ref{R:const}.

For the monomorphism $m_\lambda \colon F^\lambda 0 \to \nu F$ put
$$
\mu F = m_\lambda [F^\lambda 0] \subseteq \nu F\,.
$$
Choose an isomorphism $r\colon \mu F \to F^\lambda 0$  such that $m=m_\lambda \cdot r \colon \mu F \to \nu F$ is the inclusion map. Then there exists a unique algebra structure $\iota \colon F(\mu F) \to \mu F$ for which $r$ is an isomorphism of algebras:
$$
r\colon (\mu F, \iota) \xrightarrow{\ \sim\ } (F^\lambda 0, w_{\lambda, \lambda +1}^{-1})\,.
$$
The following commutative diagram
$$
\xymatrix@C=5pc@R=3pc{ 
\mu F \ar[r]^{\iota^{-1}} \ar[d]_r & F(\mu F) \ar[d]^{Fr}\\
F^\lambda 0 \ar[r]^{w_{\lambda, \lambda+1}}\ar[d]_{m_\lambda} & F(F^\lambda 0) \ar[dl]_{m_{\lambda+1}}
\ar[d]^{Fm_\lambda}\\
\nu F \ar[r]_\tau & F(\nu F)
} 
$$
proves that $m=m_\lambda \cdot r$ is the unique coalgebra  homomorphism, as required.

\vskip 1mm
(2) Let $F$ be arbitrary. We can assume $F\emptyset \neq \emptyset$. Our proposition holds  for the functor $G$ in Remark~\ref{R:const}(e). Since $F$ and $G$ agree an all nonempty sets
and $F\emptyset \ne \emptyset\ne G\emptyset$, they have the same terminal coalgebras. Since the initial algebra of $G$ is, as we  have just seen, obtained via the initial-algebra chain, and $F$ has from $\omega$ onwards the same initial-algebra chain, $F$ and $G$ have the same initial algebras. Thus, our proposition holds for $F$ too.
\end{proof}

Barr \cite{B} calls a functor \textit{continuous} if it preserves limits of $\omega^{\op}$-chains. Below we call a cone 
of an $\omega^{\op}$-chain a \textit{prelimit} if every other cone has a most one factorization through it. That is,
a prelimit is a collectively monic cone.

\begin{defi} A set functor is called \textit{{precontinuous}} if it preserves nonempty prelimits of $\omega^{\op}$-chains.
\end{defi}

\begin{exa}	\label{E:prec}
	
	\indent	(1) All finitary set functors are precontinuous.
	To verify this, we can restrict ourselves to set functors $F$ preserving finite
		intersections and inclusion and distinct from $C_0$, the constant functor of
		value $\emptyset$. In fact, the case $C_0$ is trivial, and for every
		other functor use $G$ of Remark 3.1(e): since $F$ is finitary, so is $G$.
		
		Let $a_n \colon A_{n+1} \to A_n$ be an $\omega^{\op}$-chain
		with a prelimit $b_n \colon B \to A_n$. Given distinct elements $x,x'$ of $FB$, we are to find $n$ with 
		$$
		Fb_n(x) \neq Fb_n(x').
		$$
		Since $F$ is finitary and preserves inclusion, there is a finite subset $C \subseteq B$
		with $x,x'$ lying in $FC$. The cone $(b_n)$ is collectively monic, thus there exists $n$ such that the restriction
		of $b_n$ to $C$ is monic. $F$ preserves (finite intersections, thus) monomorphisms, hence $Fb_n$ has the 
		desired propery: its restriction to $FC$ is monic.

	\indent (2) All continuous set functors are of course precontinuous. Example: polynomial functors for infinitary signatures.
	
		\indent (3)  Composites, products and coproducts of precontinuous functors are clearly precontinuous.
	
	\indent (4) Subfunctors of a precontinuous functor $F$ are precontinuous. Indeed, let $\mu \colon G \to F$ be a natural
	mono-transformation. And let $a_n \colon A_{n+1} \to A_n$ be an $\omega^{\op}$-chain with a non-empty prelimit $b_n \colon B \to A_n$. Since the cone $(Fb_n)$ is collectively monic, so is $((Fb_n) \cdot \mu_B)$. By naturality the last cone is $(\mu_{A_n} \cdot Gb_n)$.
	Thus also $(Gb_n)$ is collectively monic, as required.
	
	\indent (5) The functor $D$ of discrete probability distributions is precontinuous. Recall that a discrete probability distribution
on a set $X$ is given by a function $p \colon  X \to [0, 1]$ such $\sum_{x \in X}p(x)=1$. (Thus all $p(x)$ but countably many are $0$.) This extends to a probability distribution on $ X$ by $\mu (M)= \sum _{m \in M}  p(m)$ for all $M \subseteq X$.The functor $D$ assing to a set the set of its discrete probability distributions.
Given a function $f \colon X \to Y$, then $Df $ assigns to a distribution $\mu$ on $X$ the distribution $M \mapsto \mu (f^{-1} (M))$
for all $M \subseteq Y$. For an argument why $D$ preserves prelimits of
$\omega^{\op}$-chains see \cite{W}, Example 15.

	\indent (6) The functor $FX=(DX)^A$ is precontinous: it is a composite of $D$ and the 
	continuous functor $(-)^A$. Its coalgebras are probabilistic labelled transition systems with the set $A$ of actions.

\end{exa}

\begin{rem}
Restricting ourselves to \textit{nonempty} prelimits in the above definition is needed: finitary functors do not preserve prelimits
of $\omega^{\op}$-chains in general. Consider the functor $C_{2,1}$ taking the empty set to $2$ and all other sets to $1$.
\end{rem}

The following proof is based on ideas of Worrell \cite{W}. Recall the connecting maps $v_{i,j}$ of the terminal-coalgebra chain from Remark \ref{R:const}(d).

%
%

\begin{thm}
	For every precontinuous set functor $F$ the terminal-coalgebra chain converges in $\delta \geq \omega$
	steps and the connecting map
	$$
	v_{\delta, \omega} \colon \nu F \to F^\omega 1
	$$
	is monic.
\end{thm}

\begin{proof}
	(1) We can restrict ourselves to precontinuous  set functors preserving finite
	intersections (thus preserving monomorphisms) and inclusion and distinct from $C_0$, the constant functor of
	value $\emptyset$. In fact, the theorem is trivial for $C_0$, and for every
	other set functor $F$ all ordinals $i$ fulfil $F^i 1 \neq \emptyset$ by
	\cite{W}, Lemma 6. Thus the functor $G$ of Remark 3.1(e) has the same
	terminal-coalgebra chain as $F$. And it is precontinuous, since $F$ is.
	
	(2) The connecting
	morphisms $v_{\omega, n}: F^{\omega} 1 \to F^n 1$ for $n<\omega$ form a prelimit that $F$ preserves: the cone of $Fv_{\omega, n}=v_{\omega +1, n+1}$
	is collectively monic. Thus the factorizing morphism $v_{\omega +1, \omega}$
	of that cone is monic: recall from Remark \ref{R:const} that is it defined by the composites $v_{\omega +1, \omega} \cdot v_{\omega, n+1}=Fv_{\omega,n}$ for all $n<\omega$. 
	
	It follows that for every infinite ordinal $i$ the connecting map $v_{i,\omega}$ is monic. Let us verify it by transfinite induction.
	We have seen that this holds for $\omega +1$. If this holds for $i$ it holds also for $i+1$: We know that $Fv_{i,\omega} =v_{i+1,\omega +1}$ is monic, hence so is \
	$v_{i+1, \omega} = v_{\omega +1,\omega} \cdot v_{i, \omega +1}$, as required. Limit steps are trivial: a limit cone of a chain of monomorphisms has monic limit projections. And because $i > \omega$, the limit does not change when the first $\omega$ steps are deleted from the chain.
	
	Since the set $F^\omega 1$ has only a set of subobjects, the nonincreasing chain of subobjects $v_i, \omega$ stops after $\delta$
steps for some ordinal $\delta$, as claimed.

\end{proof}

\begin{nota}
  For a precontinuous set functor denote by $\delta$ the
  smallest infinite ordinal at which the terminal-coalgebra chain converges.
\end{nota}

\begin{rem} (1) For a precontinuous functor $F$ the initial algebra exists and has the form $\mu F=F^i 0$
for some ordinal $i \geq \delta$ (Remark \ref{R:const}(b)). Thus we have the subobject $m: F^i 0 \to F^\delta 1$ of Proposition
 4.1. The fact that this is a coalgebra homomorphism means that
$$
m=v_{\delta + 1, \delta} \cdot Fm \cdot w_{i,i+1}.
$$

	(2) All the examples of the precontinuous functors in 4.3 preserve nonempty countable intersections. (For finitary
	functors the verification is analogous to the proof of Proposition 10 in \cite{W}.) It then follows that we can choose
	$$
	\delta = \omega + \omega.
	$$
	Indeed, the limit defining $F^{\omega + \omega} 1$ is just the intersection of the subobjects $v_{\omega +i, \omega}$ for $i < \omega$.
	Thus $F$ preserves that limit which means that the terminal-coalgebra chain converges in  $\omega + \omega$ steps.

(3) Recall the morphism $\bar u$ from Remark 3.1(f):
	
\end{rem}

\begin{lem}\label{L:uu} Every precontinuous functor fulfils $\bar u = v_{\delta, \omega} \cdot m \cdot w_{\omega,i}.$
\end{lem} 

\begin{proof} We prove that the squares defining $\bar u$ in 3.1(f) commute when we substitute the right-hand
	side of our equation for $\bar u$. That is, we prove
$$ 
	v_{\omega, n} \cdot [ v_{\delta, \omega} \cdot m \cdot w_{\omega,i}] \cdot w_{n,\omega} = F^n ! : F^n 0 \to F^n 1
$$
This simplifies to 
$$
	v_{\delta,n} \cdot m \cdot w_{n, i} = F^n ! : F^n 0 \to F^n 1
$$
The proof is by induction. The case $n=0$ is clear. Assuming the equation holds for $n$ we prove it for $n+1$.  Using 4.6(1) we get
$$
	v_{\delta,n+1} \cdot m\cdot w_{n+1 ,i}=
	v_{\delta ,n+1} \cdot v_{\delta +1, \delta} \cdot Fm \cdot w_{i,i+1}\cdot w_{n+1 ,i}.
$$
This simplifies to

$$
 	v_{\delta +1,n+1} \cdot Fm\cdot w_{n+1 ,i+1}=	Fv_{\delta ,n} \cdot Fm \cdot Fw_{n, i}
$$
which is equal to  $ F^{n+1} !$: apply $F$ to the induction hypothesis.
	\end{proof}

 \begin{rem}\label{C:W}
%
%
%
For every precontinuous functor
 $\nu F$ is a canonical subset of $F^\omega 1 $ via $v_{\delta, \omega}$. And this endows $\nu F$ with a canonical ultrametric, as our next lemma explains.
 Recall that a metric  $d$ is called an \textit{ultrametric} if for all elements $x$, $y$, $z$ the triangle inequality can be strengthened to $d(x,z)\leq \max (d(x,y), d(y,z))$.
 \end{rem}

\begin{lem}\label{L:fin}
Every limit $L$ of an $\omega^{\op}$-chain  in $\Set$ carries a complete ultrametric: assign to $x\ne y$ in $L$ the distance $2^{-n}$ where $n$ is the least natural number such that the corresponding  limit projection separates $x$ and $y$.
\end{lem}

\begin{proof}
Let $l_n\colon L\to  A_n$ ($n\in \N$) be a limit cone of an $\omega^{\op}$-chain $a_n\colon A_{n+1} \to A_n$ ($n\in \mathbb{N}$). For the above function
$$
d(x,y) = 2^{-n}
$$
where  $l_n(x) \ne l_n(y)$ and $n$ is the least such number  we see that $d$ is symmetric. It satisfies the ultrametric inequality
$$
d(x,z) \leq \max \big(d(x,y), d(y,z)\big) \quad \mbox{for all\ }\ x,y,z \in L\,.
$$
This is obvious if the three elements are not pairwise distinct. If they are, the inequality follows from the  fact that if $l_n$ separates  two elements, then so do all $l_m$ with $m\geq n$.

 It remains to prove that the space $(L,d)$ is  complete. Given a Cauchy sequence $x_r \in L$ ($r\in \N$), then for every $k\in \N$ there exists $r(k) \in \N$ with
$$
d(x_{r(k)}, x_n) < 2^{-k} \quad \mbox{for all} \ \ n\geq r(k)\,.
$$
Choose $r(k)$'s to form an increasing sequence. Then  $d(x_{r(k)}, x_{r(k+1)}) < 2^{-k}$, i.e., $l_k(x_{r(k)}) = l_k(x_{r(k+1)})$. Therefore, the elements $y_k=l_k(x_{r(k)})$ are compatible: we have $a_{k+1}(y_{k+1}) = y_k$ for all $k\in \N$. Consequently, there exists a unique $y\in L$ with $l_k(y) = y_k$ for all $k\in \N$. That is, $d(y, x_{r(k)})< 2^{-k}$. Thus, $y$ is the  desired limit:
\[
y=\lim_{k\to\infty} x_{r(k)} \quad \mbox{implies}\quad y=\lim_{n\to \infty} x_n\,.
\tag*{\qedhere}
\]
\end{proof}

We conclude  that for a finitary set functor both $\nu F$ and $\mu F$ carry a canonical  ultrametric: $\nu F$ as a subspace of $F^\omega 1$ via $v_{\omega +\omega, \omega }\colon \nu F \to F^\omega 1$, and $\mu F$ as a subspace  of $\nu F$ via $m$. Given $t\ne s$ in $\nu F$ we have $d(t,s) = 2^{-n}$ for the least $n\in \N$ with $v_{\omega +\omega, n} (t) \ne v_{\omega +\omega, n} (s)$.
The isomorphism $\tau \colon \nu F \xrightarrow{\ \sim\ } F(\nu F)$ then makes $F(\nu F)$ also a canonical ultrametric space, analogously for $F(\mu F)$.

\begin{nota}\label{N:p}
Given a precontinuous set functor $F$ with $F\emptyset \ne \emptyset$, choose an element
$$
p\colon 1\to F\emptyset\,.
$$
This defines the following morphisms for every $n\in \N$:
\begin{align*}
e_n &\equiv F^n 1 \xrightarrow{\ F^n p\ } F^{n+1}0 \xrightarrow{\ w_{n+1, \omega}\ }
F^\omega 0 \xrightarrow{\ \bar u\ } F^\omega 1\\
\intertext{and}
\varepsilon_n &= e_n \cdot v_{\omega, n} \colon F^\omega 1 \to F^\omega 1\,.
\end{align*}
\end{nota}

\begin{obs}\label{O:various}
\hfill

(a) For every $n\in \N$ we have a commutative square below
$$
\xymatrix@C=4pc{
F^n 1 \ar[r]^{F^np} \ar@{=}[d] & F^{n+1} 0\ar[d]^{F^{n+1} !}\\
F^n 1 & F^{n+1} 1 \ar [l]^{v_{n+1, n}}
} 
$$
This is obvious for $n=0$. The induction step just applies $F$ to the  given square.

\quad (b)\qquad   $v_{\omega,n} \cdot e_n =\id_{F^n 1}$\,.

  Indeed, in the following diagram
  $$
\xymatrix@C=4pc{
F^n 1 \ar[r]^{F^np} \ar@{=}[dd] & 
F^{n+1} 0\ar[r]^{w_{n+1, \omega}} \ar[dr]^{F^{n+1}!}& F^\omega 0 \ar[r]^{\bar u} & 
F^\omega 1  \ar[dd]^{v_{\omega, n}} \ar[dl]^{v_{\omega, n+1}} \\
&&F^{n+1} 1 \ar[dr]^{v_{n+1,n}}&\\
F^n 1 \ar@{=}[rrr] &&& F^n 1
} 
$$
the upper right-hand part commutes by the definition of $\bar u$, see Remark~\ref{R:const}(f),  the left-hand one does by (a), and the lower right-hand triangle is clear.

\quad (c) \qquad $v_{\omega,n} \cdot \varepsilon_n = v_{\omega,n}$\,.

This follows from (b): precompose it with $v_{\omega,n}$.

\quad (d) \qquad $\varepsilon_n\cdot \varepsilon_{n+1}=\varepsilon_n$. Indeed, we have
\begin{align*}
\varepsilon_n\cdot \varepsilon_{n+1} &  = \bar u \cdot w_{n+1, n} \cdot F^np \cdot v_{\omega, n} \cdot \varepsilon_{n+1} && \mbox{def. of $\varepsilon_n$}\\
&= \bar u  \cdot w_{n+1,n} \cdot F^n p\cdot (v_{n+1}\cdot v_{\omega, n+1}) \cdot \varepsilon_{n+1}&&\\
&= \bar u  \cdot w_{n+1,n} \cdot F^n p\cdot v_{n+1}\cdot v_{\omega, n+1} \cdot \varepsilon_{n+1}&& \mbox{by (c)}\\
&= \bar u  \cdot w_{n+1,n} \cdot F^n p\cdot v_{\omega, n}&&\\
&=\varepsilon_n\,.
\end{align*}

\end{obs}

\begin{thm}\label{T:fin} 
For a precontinuous set functor $F$ with $F \emptyset \ne \emptyset$ the Cauchy completions of the ultrametric spaces $\mu F$ and $\nu F$ coincide. And the algebra structure $\iota$ of $\mu F$ determines the coalgebra structure $\tau$ of $\nu F$  as the unique continuous extension of $\iota^{-1}$.

\end{thm}

\begin{proof}
Assume first that $F$ preserves inclusion (Remark \ref{R:const}(e)). 

(a) We prove that the subset
$$
\bar u = v_{\delta, \omega} \cdot m \cdot w_{\omega ,i} \colon \mu F \to F^\omega 1
$$
 of Lemma~\ref{L:uu} is dense in $F^\omega 1$, thus, the complete space $F^\omega 1$ is a Cauchy completion of both $\bar u [\mu F]$ and  $v_{\delta, \omega} [\nu F]$.
 
 For every $x\in F^\omega 1$ the sequence $\varepsilon_n(x)$ lies in the image of $e_n\cdot v_{\omega,n}$ which, in view of the definition of $e_n$, is  a subset of the image  of $\bar u$. And we have 
 $$
 x=\lim_{n\to\infty} \varepsilon _n(x)
 $$
 because Observation~\ref{O:various} (c) yields $v_{\omega,n} (x) = v_{\omega,n}(\varepsilon_n(x))$, thus
 $$
 d\big(x, \varepsilon_n(x)\big) < 2^{-n}\quad \mbox{for all\ }\ n\in \N\,.
 $$
 
 (b) 
 The continuous map $\iota^{-1}$ has at most one continuous extension to $\nu F$. And $\tau$ is such an extension: it is not only continuous, it is an isometry. And it extends $\iota^{-1}$ by Proposition~\ref{P:sub}: 
$$
\xymatrix@C=4pc{
\mu F \ar[r]^{\iota^{-1}}
\ar@{_{(}->} [d]_{m} & F(\mu F) \ar@{_{(}->}[d]^{Fm}\\
\nu F \ar[r]_\tau & F(\nu F)
} 
$$
Since $Fm$ is an inclusion map, $\tau$ is an extension of $\iota^{-1}$.
 
Once we have established (a) and (b) for 
inclusion-preserving finitary 
 functors, it holds for all finitary set functors $F$. Indeed, 
  there exists an inclusion-preserving set functor $G$ that agrees with $F$ on all nonempty sets and functions, and fulfils $G\emptyset\ne\emptyset$, see Remark~\ref{R:const}(e). Consequently, the coalgebras for $F$ and $G$ coincide. And  the initial-algebra chains coincide from $\omega$ onwards, in particular, we can assume $F^i 0 = G^i 0$, that is, $F$ and $G$ have the same initial algebra.
 \end{proof}

 \begin{exa}\label{ex4.7}
  For the finite power-set functor $\cp_f$ the initial algebra can be described as
 $$
 \mu \cp_f= \mbox {\ all finite extensional trees,}
 $$
 (where trees are considered up to isomorphism). Recall that  a tree is called \textit{extensional} if for every node $x$ the maximum subtrees of $x$ are pairwise  non-isomorphic. And  it is called \textit{strongly extensional} if it has no nontrivial tree bisimulation. 
(A tree bisimulation is a bisimulation $R$ on the tree which relates the root with itself and relates only vertices of the same height.) 
For finite trees these two concepts are equivalent. Worrell proved in \cite{W} that the terminal coalgebra can be described as follows:
 $$
 \nu \cp_f = \mbox{\  all finitely branching strongly extensional trees}
 $$
 with the coalgebra structure inverse to tree tupling.
 Whereas the terminal-coalgebra chain yields
 $$
 \cp_f^\omega 1 = \mbox{\ all strongly extensional trees}
$$
  The metric on $\cp_f^\omega 1$ assigns to trees $t\ne s$ 
the distance $d(t,s)= 2^{-n}$, where $n$  is the least number with $\partial_n t \ne \partial_n s$. Here $\partial_n t$ is the extensional tree obtained from $t$ by 
cutting it at level $n$ and forming the extensional quotient of the resulting tree.

\end{exa}

\begin{cor}\label{C:fin} 
Assuming GCH, let $F$ be a precontinuous set functor whose initial algebra has an uncountable regular cardinality. Then the terminal coalgebra has the same cardinality.
Shortly,
$$
\mu F \simeq \nu F\,.
$$
\end{cor}

Indeed, let $\lambda = |\mu F|$. The Cauchy completion of $\mu F$ has power at most $\lambda^\omega =\lambda$ (see Remark~\ref{R:lambda}(d)), thus, $|\nu F|\leq \lambda$ by the above theorem. And $|\nu F|\geq \lambda$ follows  from Proposition~\ref{P:sub}.

\begin {rem} 
The above corollary does not extend to $\aleph_0$: The finitary set functor $FX= X \times 2+1$ has a countable initial algebra (of finite binary sequences) and an uncountable
terminal coalgebra (of finite and infinite binary sequences).

\end{rem}
\begin{exa}\label{E:tree}
 Let $\Sigma$ be a (possibly infinitary) signature
with at least one nullary symbol. We choose one, $\perp \in \Sigma_0$. 
 Recall that a \textit{$\Sigma$-tree} is an ordered tree labelled in 
 $\Sigma$ so that a node with  a label $\sigma \in \Sigma_n$ has precisely $n$ successor nodes. We  consider these trees again up to isomorphism. A $\Sigma$-tree is called \textit{well-founded} if every branch of it is finite.

By Theorem II.3.7 of \cite{AT} the  polynomial functor  $H_\Sigma X = \coprod \Sigma_n \times X^n$ has the initial algebra
\begin{align*}
\mu H_\Sigma & = \mbox{\ all well-founded $\Sigma$-trees}\\
\intertext{and the terminal coalgebra}
\nu H_\Sigma &= \mbox{\ all $\Sigma$-trees.}
\end{align*}
Moreover, $H_\Sigma$ preserves limits of $\omega^{\op}$-chains, hence, $\nu H_\Sigma = H_\Sigma^\omega 1$. The complete ultrametric on $\nu H_\Sigma$ assigns to $\Sigma$-trees $t\ne s$ the distance $d(t,s)= 2^{-n}$, where $n$ is the least number with $\partial_n t\ne \partial_ns$.  Here $\partial_nt$ is the $\Sigma$-tree obtained  cutting $t$ at height $n$ and  relabelling all leaves of height $n$  by $\perp$. We conclude that
$$
\nu H_\Sigma \quad \mbox { is the Cauchy completion of\ \, $ \mu H_\Sigma $}.
$$
Consequently, if the set of all well-founded trees has a regular uncountable cardinality, then this is also the cardinality of the set of all $\Sigma$-trees.
 
\end{exa}

The preceding example  only works if $\Sigma_0\ne \emptyset$. In case $\Sigma_0 =\emptyset$ we have $H_\Sigma \emptyset = \emptyset$ and Theorem 4.7 does not apply.

\section{Precontinuous Set Functors and CPO's} 

Analogously to the preceding section we prove that given a precontinuous set functor $F$ with $F\emptyset \ne \emptyset$ both $\nu F$ and $\mu F$ carry a canonical partial order with a common ideal $cpo$-completion. Again, the coalgebra structure $\tau\colon \nu F \to F(\nu F)$ is the unique continuous extension of $\iota^{-1}$.

\begin{nota}\label{N4.8}
Given a precontinuous set functor $F$ with $F\emptyset \ne \emptyset$ choose an element
$$
p\colon 1 \to F\emptyset\,.
$$
This   defines a partial order  $\sqsubseteq$ on the set $F^\omega 1$ as follows:
 $$
t \sqsubseteq s \quad \mbox{if}\ \ t=\varepsilon _n(s) \quad  \mbox{for some}\ \ n\in \N
 $$
for $t\ne s$ in $F^\omega 1$ (see Notation~\ref{N:p}).
 \end{nota}

 \begin{thmC}[{\cite[Theorem 3.3]{A0}}]\label{T:cpo}
 $(F^\omega 1, \sqsubseteq)$ is a $cpo$, i.e., every directed subset has a join. Moreover, every strictly increasing  $\omega$-chain in $F^\omega 1$ has a unique upper bound.
\end{thmC}

The theorem in \cite{A0} is formulated for continuous functors for which the statement is that $\nu F$ is a $cpo$. However, the proof remains the same for precontinuous functors if we formulate the result as above.

\begin{exa}\label{E:4.10}
\hfill

(1) For $\cp_f$, given strongly extensional trees $t\ne s$, then $t \sqsubseteq s$ iff $t=\partial_ns$ for some $n$.

(2) The functor $FX = \Sigma \times X+1$ has a terminal coalgebra
$
\nu F= F^\omega 1 = \Sigma ^\ast + \Sigma^\omega.
$
 For words $t\ne s$ we have
$$
t \sqsubseteq s \quad \mbox{iff \ $t$ \ is a prefix of \  $s$}.
$$

(3) More generally, given a signature $\Sigma$, for $\Sigma$-trees $t\ne s$ we have $t \sqsubseteq s$ iff $t$ is a cutting of $s$, i.e., $t=\partial_n s$ for some $n$ (see Example \ref{E:tree}).
\end{exa}

We conclude  that for a precontinuous set functor $F$ with $F\emptyset \ne \emptyset$  both $\nu F$ and $\mu F$ carry a canonical partial order: $\nu F$ as a subposet of $F^\omega 1$ via $v_{\delta, \omega}$ and $\mu F$ as a subposet of $\nu F$ via $m$. For  $t\ne s$ in $\nu F$ we have
$$
t \sqsubseteq s \quad \mbox{iff \ } v_{\delta, \omega}(t) = v_{\delta, \omega } \cdot \varepsilon_n(s) \ \mbox{for some}\ \ n\in \N\,.
$$

This partial order depends on the choice of an element $p$ of $F \emptyset$.

\begin{rem}\label{R4.11} Recall the concept of \textit{ideal completion} of a poset {P}.
%
 This is a $cpo$ $\bar P$ containing $P$ as a subposet such that for every monotone function $f\colon P\to Q$, where $Q$ is a $cpo$, there exists a unique continuous extension $f\colon \bar P\to Q$.
\end{rem}

\begin{thm}\label{T4.12}
For a precontinuous set functor $F$ with $F \emptyset \ne \emptyset$ the ideal completions of the posets $\mu F$ and $\nu F$ coincide. And the algebra structure $\iota$  of $\mu F$ determines the coalgebra structure $\tau$ of $\nu F$  as the unique continuous extension of $\iota^{-1}$.
\end{thm}

\begin{proof}
As in the proof of Theorem~\ref{T:fin},  we can assume that $F$ is preserves inclusion. We prove that $F^\omega 1$ is an ideal completion of $\mu F$ and $\nu F$.

\vskip 1mm
(1) Given $x\in F^\omega 1$, we find an $\omega$-sequence in $\mu F$ with join $x$. In fact,  the sequence $\varepsilon_n(x)$ lies in the image of  $\bar u\colon \mu F \hookrightarrow F^\omega 1$ (see the definition of $\varepsilon_n$). This is an $\omega$-sequence:
for every $n\in\N$ we have
$$
\varepsilon_n(x) \sqsubseteq \varepsilon_{n+1} (x) \sqsubseteq x\,. 
$$
Recall from Theorem~\ref{T:cpo}  that strict $\omega$-sequences have unique upper bounds in $F^\omega 1$, thus,
$$
x = \bigsqcup_{n\in \N} \varepsilon_n (x)\,.
$$

\vskip 1mm
(2) Let $D\subseteq F^\omega 1$ be a directed set with $ x =\bigsqcup D$. If $x \notin D$, then $D$ is cofinal with the sequence $\varepsilon_n (x)$, $n\in\mathbb{N}$. Indeed, every $d\in D$ fulfils $d \sqsubseteq x$, that is, $d = \varepsilon_n (x)$ for  some $n$. The rest follows from (1).

\vskip 1mm
(3) $F^\omega 1$ is an ideal completion of $\mu F$. Indeed, let $Q$ be a $cpo$ and $f\colon \mu F \to Q$ a monotone function. Extend it to $F^\omega 1$ by putting, for every $x\in F^\omega 1 - \bar u[\mu F]$
$$
\bar f(x) =\bigcup_{n\in \mathbb{N}} f\big(\varepsilon_n(x)\big)\,.
$$
This map is continuous due to (2), and the extension is unique due to (1).

\vskip 1mm
(4) $F^\omega 1$ is also an ideal completion of $\nu F$. The argument is analogous.


\vskip 1 mm
(5) The proof that $\tau$ is the unique continuous extension of $\iota^{-1}$ is completely analogous to the metric case in Theorem~\ref{T:fin}.
\end{proof}

 \section{Unexpected Finitary Endofunctors} 
 
 We know from the proof of Theorem~\ref{T:pres} that all endofunctors of $\Set_{\leq \lambda}$ 
preserve colimits of $\lambda$-chains. In particular, all endofunctors of $\Set_{\leq \omega}$ (the category of countable sets) are finitary. Nevertheless, this category is not algebraically cocomplete, as  shown in Example~\ref{E:new}. In this section we turn to singular cardinals, e.g. $\aleph_\omega = \bigvee\limits_{n<\omega} \aleph_n$ with countable cofinalities (see Remark~\ref{R:lambda}). The category $\Set_{\leq \aleph_\omega}$ is also  not algebraically cocomplete, as we demonstrate in the next example. We are going  nonetheless to  prove that all endofunctors of $\Set_{\leq\aleph_\omega}$ are finitary, i.e., they preserve all existing filtered colimits. Below  we use ideas of Section 4.6 of \cite{AT}. Here is a surprisingly simple functor which is  finitary but has no terminal coalgebra:

 \begin{exa}\label{E:omega}
 The endofunctor 
 $$ FX = \aleph_\omega \times  X
 $$
 of $\Set_{\leq \aleph_\omega}$ (a coproduct of $\aleph_\omega$ copies of $\Id$) does not have a terminal coalgebra. 
We use
 the fact that
 $$
 \big(\aleph_\omega\big)^\omega > \aleph_\omega
 $$
 see Remark~\ref{R:lambda}(b). 
The proof that $\nu F$ does not exist in $\Set_{\leq \aleph_{\omega}}$ is similar to that of Example~\ref{E:new}.
Suppose that $\tau \colon T \to \aleph_\omega \times T$ is a terminal coalgebra. For every function $f\colon \N \to \aleph_\omega$ define a coalgebra $\hat f\colon \N \to \aleph_\omega \times \N$ by
 $$
 \hat f (n) = \big(f(n) , n+1\big) \quad \mbox{for} \quad n\in \N\,.
 $$
 We obtain a unique homomorphism $h_f$:
 $$
\xymatrix@C=4pc{
\N \ar[d]_{h_f} \ar[r]^{\hat f\ \ } & \aleph_\omega \times \N\ar[d]^{\id\times h_f}\\
T \ar [r]_{\tau\ \ } & \aleph_\omega \times T
}
$$
We claim that for all pairs of functions $f$, $g\colon \N\to \aleph_\omega$ we have
$$
h_f (0) = h_g(0) \quad \mbox{implies} \quad f=g\,.
$$
Indeed, since $\tau\cdot h_f(0)  = \tau \cdot h_g(0)$, the square above yields
$$
\big(f(0), h_f(1)\big) = \big( g(0), h_g(1)\big)\,,
$$
in other words
$$
h_f(1) = h_g(1) \quad \mbox{and} \quad f(0) = g(0)\,.
$$
The first of these equations yields, via the above square,
$$
h_f(2) = h_g(2) \quad \mbox{and}\quad f(1) = g(1)\,,
$$
etc.

This is the desired contradiction: since for all functions $f\colon \N \to \aleph_\omega$ the elements $h_f(0)$ of $T$ are pairwise distinct, we get
$$
|T| \geq |\aleph_\omega^{\N} | > \aleph_\omega\,.
$$
\end{exa}

\begin{rem}\label{R:omega}
\hfill

(a) Recall that a \textit{filter\/} on a set $X$ is a collection of nonempty subsets of $X$
closed upwards and closed under finite intersections. The principle filter is one containing a singleton set. Maximum filters are called \textit{ultrafilters\/} and are characterized by the property that for every $M\subseteq X$ either $M$ or its complement  is a member.

\vskip 1mm
(b) Recall further that a cardinal $\alpha$ is \textit{measurable\/} if there exists a nonprinciple ultrafilter on $\alpha$ closed under intersections of less than $\alpha$ members. Every measurable cardinal is inaccessible, see e.g. \cite{J}, Lemma 5.27.2. Therefore, the assumption that no measurable  cardinal exists is consistent  with ZFC set theory. This also implies that every measurable cardinal is larger than $\aleph_\omega$.

\vskip 1mm
(c) To say that a cardinal $\alpha$ is not measurable is equivalent to saying that for every nonprinciple ultrafilter on $\alpha$ contains  members $U_0$, $U_1$, $U_2, \dots$ such that $\bigcap\limits_{k<\omega}U_k$ is no member.

\end{rem}

\begin{thm}\label{T:omega}
Let $\lambda$ be a cardinal of cofinality $\omega$ such that no measurable cardinal is smaller or equal to $\lambda$ (e.g. $\lambda =\aleph_\omega$). Then every endofunctor of $\Set_{\leq \lambda}$ preserves filtered colimits.
\end{thm}

 \begin{proof}\hfill
 
(1) $\Set_{\leq \lambda}$ has colimits of $\lambda$-chains and every endofunctor preserves them.

Indeed, let a chain with objects $A_i$ and  connecting morphisms $a_{i,j}$ ($i\leq j<\lambda$) be given  in $\Set_{\leq\lambda}$. If $a_i\colon A_i\to A$ ($i<\lambda$) is a colimit in $\Set$, then this is also a colimit in $\Set_{\leq\lambda}$ because
$$
|A| \leq  \sum_{i<\lambda} |A_i| \leq \lambda^2 =\lambda\,.
$$

For every endofunctor $F$ we prove that $(Fa_i)_{i<\lambda}$ is also a colimit cocone in $\Set$ (thus in $\Set_{\leq \lambda}$). This is equivalent to proving that the cocone has properties (i) and (ii) of Remark~\ref{R:gener}. That is
\begin{enumerate}
\item[(i)]$FA$ is the union of the images of $Fa_i$,
\end{enumerate}
and
\begin{enumerate}
\item[(ii)] given $i<\lambda$, every pair of elements of $FA_i$ that $Fa_i$ merges is also merged by $Fa_{ij}$ for some connecting morphism $a_{i,j} \colon A_i \to A_j$.
\end{enumerate}

For (i) use Proposition~\ref{P:lambda} which, by Remark \ref{R:nova}, applies to nonregular cardinals: given  $b\in FA$ there exists a subset $c\colon C\hookrightarrow A$ with $|C|<\lambda$ such that $b\in Fc[FC]$. It follows that the subset fulfils $c\subseteq a_i$ for some $i<\lambda$, hence $b\in Fa_i[FA_i]$.

For (ii), let $x_1$, $x_2\in FA_i$ fulfil $Fa_i(x_1) = Fa_i(x_2)$. As above, there exists a subset $c\colon C\hookrightarrow A_i$ with  $|C|<\lambda$ such that $x_1$, $x_2\in Fc[FC]$. Since $|C\times C| <\lambda$, we can find an ordinal $j$ with $i\leq j<\lambda$ such that every pair in $C\times C$ merged by $a_i$ is also merged by $a_{i,j}$. In other words
$$
\ker a_i\cdot c \subseteq \ker a_{i,j} \cdot c\,.
$$
Consequently, $a_i\cdot c$ factorizes through $a_{i,j}\cdot c$:
$$
\xymatrix {
C\ar@{_{(}->} [d]_{c}& & \\
A_i \ar[rr]^{a_{i,j}}
\ar[dr]_{a_i} &&
A_j \ar@<-1ex>[dl]^{a_j}\\
& A \ar@<2ex>@{-->}[ur]^{f}
&
}
$$
From $a_{i,j}\cdot c = f\cdot a_i \cdot c$ we derive, given $y_k\in FC$ with $x_k=Fc(y_k)$, that
$$
Fa_{i,j} \cdot (x_1) = F a_{i,j} \big(Fc(y_1)\big) =Ff \big(Fa_i (x_1)\big)
$$
and analogously for $x_2$. Thus, $F a_i(x_1)= Fa_i(x_2)$ implies $F a_{i,j} (x_1) = F a_{i,j} (x_2)$, as desired.

\vskip 1mm
(2)  $\Set_{\leq \lambda}$ is closed under existing filtered colimits in $\Set$. It is sufficient to prove this for  existing directed colimits due to Theorem~5 of \cite{AR}. Let $D\colon (I, \leq) \to \Set_{\leq \lambda}$ be a directed diagram with a colimit $a_i\colon A_i \to A$ ($i\in I$) in $\Set_{\leq \lambda}$. We prove  properties (i) and (ii) of Remark~\ref{R:gener}.

Ad(i): for the union $m\colon A'\hookrightarrow A$ of images of all  $a_i$ we are to prove $A'=A$. Factorize $a_i= m\cdot a_i'$ for $ a_i'\colon A_i \to A^\prime$ ($i\in I$), then $(a_i')_{i\in I}$ is clearly a cocone of $D$. Thus, we have $f\colon A\to A^\prime$ with $a_i^\prime = f. d_i$ ($i\in I$). From $(m.f).a_i=a_i$ for all $i\in I$ we deduce $m.f=\id$, thus, $A^\prime =A$.

Ad(ii): given $i\in I$ and elements $x$, $x^\prime \in A$ merged by $a_i$, we prove that they are merged by some connecting map $a_{i,k} \colon A_i\to A_k$ of $D$ ($k\in i$).
Without loss of generality we assume that $i$ is the least element of $I$
(if not, restrict $D$ to the upper set of $i$ which yields a directed diagram
with the same colimit). For every $k\geq i$ put $x_k=a_{i,k}(x)$ and define, given $j\in J$, a map $b_j\colon A_j \to A +\{t\}$ in an element $y\in A_j$ as follows
$$
b_j(y)=\begin{cases} t& \mbox{if} \quad a_{j,k} (x_j) = a_{j,k}(y)\, \mbox{\ for some $k\geq j$}\\
 a_j(y)  &\mbox {else}
 \end{cases}
 $$
 Since $(I, \leq)$ is directed, it is easy to see that this yields a cocone of $D$. We have a unique $f\colon A\to A+\{t\}$ with  $b_i =f.a_i$ ($i\in I$). Clearly $b_i(x) =t$, therefore, $a_i(x) = a_i(x^\prime)$ implies $b_i(x^\prime)=t$. This proves  $a_{i,k} (x) = a_{i,k}(x^\prime)$ for some $k\geq i$, as desired.
 
 \vskip 1mm
 (3) To prove the theorem, it is sufficient to verify for every endofunctor $F$ of $\Set_{\leq \lambda}$ and every object $X$ that every element $a\in FX$  in the image of $Fy$ for some finite subset $y\colon  Y\hookrightarrow X$. Indeed, since filtered colimits are, due to (2), formed as in $\Set$, the fact that $F$ preserves them is then proved completely analogously to (1) above.

First, $F$  preserves colimits of $\omega$-chains: given  an $\omega$-chain $D\colon \omega \to \Set_{\leq \lambda}$ with objects $D_i$ for $i<\omega$
 and  given cardinals $\lambda_0<\lambda_1<\lambda_2 \dots$ with $\lambda =\bigvee\limits_{n<\omega} \lambda_n$, we  define a $\lambda$-chain $\bar D \colon \lambda \to \Set_{\leq\lambda}$ by assigning to $i< \lambda_0$ the value $\bar D_i = D_0$ and to every $i$ with $\lambda_n \leq i< \lambda_{n+1}$ the value $\bar D_i = D_{n+1}$. (Analogously for the connecting morphisms). Then $D$ is cofinal in $\bar D$ (and $FD$ cofinal in $F\bar D$), thus, they have the same  colimits. Since $F$ preserves $\colim \bar D$, it also preserves $\colim D$.
 
 (4)
 For every element $a\in FX$ we are going to prove  that there exists a finite subset $y\colon Y\hookrightarrow X$ with $a\in Fy [FY]$.

 Denote by $\cf$ the collection of all subsets $y\colon Y\hookrightarrow X$ for which $a\in Fy [FY]$. We are to prove that $\cf$ contains a finite member.
Put
$$
Z=\bigcap_{Y \in \cf} Y \quad \mbox{with inclusion}\quad z\colon Z\hookrightarrow X\,.
$$

(a) First, suppose  $Z\in \cf$. We prove that $Z$ is finite, which concludes the proof.

In case $Z$ is infinite, we derive a contradiction. Express $Z$ as a union of a strictly increasing $\omega$-chain $Z=\bigcup\limits_{n<\omega} Z_n$. Obviously $Z_n\notin \cf$ (due to $Z_n\subsetneqq Z$).
 But then $F$ does not preserve the union $z =\bigcup z_n$ of the corresponding inclusion maps $z_n \colon Z_n \hookrightarrow X$ since $a\in Fz [FZ]$ but $a\notin Fz_n [FZ_n]$, a contradiction.
 
 \vskip 1mm
 (b)
 Suppose $Z\notin \cf$. We prove in part (c) that $\cf$ contains sets $V_1$, $V_2$ with $V_1\cap V_2=Z$. Since for the inclusion maps $v_i \colon V_i \to X$ we have  $a\in Fv_i [FV_i]$ but $a\notin Fz[FZ]$, we see that $F$ does not preserve the intersection $z= v_1 \cap v_2$. Thus, $Z=\emptyset$, see Corollary~\ref{C:tr}. Then we 
choose $t\in X$ and 
prove
 $$
 \{t\} \in \cf\,,
 $$
 concluding the proof. Let $h\colon X\to X$ be defined  by
 $$
 h(x) = 
\begin{cases} x & x\in V_1\\ 
t & \mbox{else}\,.
 \end{cases}
 $$
 That is, since $V_1$ and $V_2$ are disjoint, for the constant function $k\colon X\to X$ of value $t$  we have
 $$
 hv_1 = v_1 \quad \mbox{and}\quad  hv_2 = kv_2\,.
 $$
 Then we compute, given $a_i \in FV_i$ with $a= Fv_i(a_i)$:
 $$
 Fk(a) = F(kv_2) (a_2) = F(hv_2) (a_2) = Fh(a)
 $$
 as well as
 $$
 a= Fv_1(a_1) = F(hv_1)(a_1) = Fh(a)\,.
 $$
 Consequently, $Fk(a) =a$, and since for the inclusion $y\colon \{t\} \to X$ we have the following  commutative triangle
 $$
\xymatrix@C=3pc{
& \{t\} \ar [d]^{y}\\
X \ar [ur]^{!} \ar[r]_{k} & X
}
$$
this yields $a\in Fy [F\{t\}]$, as required.

\vskip 2mm
(c)
Assuming that
$$
V_1, V_2 \in \cf \quad \mbox{implies} \quad V_1 \cap V_2 \ne Z\,,
$$
we derive a contradiction. Put
$$
X_0 = X-Z
$$
 and define a filter $\cf_0$ on $X_0$ by
 $$
 Y \in \cf_0 \quad \mbox{iff}\quad Z\cup Y \in \cf\,.
 $$
 $\cf_0$ is closed under super-sets (since $\cf$ is) and  does not contain $\emptyset$ (since $Z\notin \cf$). We verify that it is closed under  finite intersection. Given $Y_1$, $Y_2 \in \cf_0$, then by our assumption  above:
  $$
  (Z\cup Y_1) \cap (Z\cup Y_2) \ne Z\,.
  $$
 This  implies that the sets $V_i = Z\cup Y_i$ are not disjoint. Thus, $F$ preserves the intersection $v_1\cap v_2$, see Corollary~\ref{C:tr}. Consequently,  $a$ lies in the $F$-image of $v_1\cap v_2$, i.e., $V_1\cap V_2 \in \cf$. Since $Y_1 \cap Y_2 = Z\cup (V_1\cap V_2)$, this yields $Y_1\cap Y_2 \in \cf_0$.
  
  By the Maximality Principle the filter $\cf_0$ is contained in an ultrafilter $\EuScript{U}$. This ultrafilter is nonprincipal: for every $x\in X_0$ we know (from our assumption $Z\notin \cf$) that $\{x\} \notin \cf$, hence, $\{x\}\notin \EuScript{U}$. Since $\card X_0$ is not measurable, there exists  by Remark~\ref{R:omega}(c) a collection
  $$
  U_n \in \EuScript{U} \ (n<\omega) \quad \mbox{with}\quad \bigcap_{n<\omega} U_n \notin \EuScript{U}\,.
  $$
  We define an $\omega$-chain $S_k \subseteq X-Z$ ($k<\omega$) of sets that are not members of $\EuScript{U}$ by the following recursion:
  $$
  S_0 =\bigcap_{n<\omega} U_n\,.
  $$
  Given $S_k$, put
$$
  S_{k+1} = S_k \cup (X_0 - U_k)\,.
  $$
  Assuming $S_k \notin \EuScript{U}$ then, since we know that $X_0 - U_k \notin \EuScript{U}$, we get $S_{k+1} \notin \EuScript{U}$ by the fact that $\EuScript{U}$ is an ultrafilter.
  
  We observe that
  $$
   X_0 = \bigcup_{k<\omega} S_k\,.
   $$
 Indeed, every element $x\in X_0$ either lies in $S_0$ or in its complement $\bigcup\limits_{n<\omega} (X - U_n)$. In the latter case we have  $x\in X - U_k\subseteq S_k$ for some $k$. We achieved
  the desired contradiction: $F$ does not preserve the $\omega$-chain colimit $X=\underset{k<\omega}{\colim} (Z \cup S_k)$. Indeed, $Z\cup S_k \notin \cf$, thus, for its embedding $s_k \colon Z\cup S_k \hookrightarrow X$ we have $a\notin Fs_k [F(Z\cup S_k)]$.
 \end{proof}

The argument used in point (1) of the above proof is analogous to that of Theorem 3.10 of \cite{AMSW}.


\section{Conclusions and Open problems} 

We have  presented a number  of categories that are algebraically complete and  cocomplete, i.e., every  endofunctor has a terminal coalgebra and an initial algebra. Examples include (for sufficiently large regular cardinals $\lambda$) the category $\Set_{\leq \lambda}$ of sets of power at most $\lambda$, $\operatorname{\bf Nom}_{\leq \lambda}$ of nominal sets of power at most $\lambda$,
 $K$-$\Vec_{\leq \lambda} $ of vector spaces of  dimension at most $\lambda$, and $G$-$\Set_{\leq \lambda}$ of $G$-sets (where $G$ is a group) of power at most $\lambda$.

All these results assumed the General Continuum Hypothesis. It is an open  question what could be proved  without this assumption. 
However set-theoretical assumptions cannot be avoided: we have seen that the Continuum Hypothesis is  \emph{equivalent}
to the algebraic cocompleteness of the category of sets of power at most $\aleph_1$.
	
We have introduced the concept of a precontinuous set functor encompassing all finitary ones, all continuous ones and composites, products and coproducts of those. For these functors $F$ with  $F\emptyset\ne \emptyset$  we have  presented a sharper result: both $\mu F$ and $\nu F$ carry a canonical partial ordering and these two posets have the same ideal $cpo$-completion. Moreover, by inverting the algebra structure of $\mu F$ we obtain the coalgebra structure of $\nu F$ as the unique continuous extension. In place  of posets and $cpo$'s, we have also got the same result for ultrametric spaces and their Cauchy completions.

For cardinals $\lambda$ of countable cofinality we have proved  that  the category $\Set_{\leq \lambda}$ is not algebraically complete, but  every endofunctor is finitary. The proof used the set-theoretical assumption that no cardinal smaller or equal to $\lambda$ is measurable. It is an open problem whether there are set-theoretical assumptions that would imply the analogous result for other cofinalities than $\aleph_0$.

\bibliographystyle{alpha}
\bibliography{Alco}

\begin{thebibliography}{AMSW19}

\bibitem[Ad{\'{a}}74]{A0}
J.~Ad{\'{a}}mek.
\newblock {Free algebras and automata realizations in the language of
  categories}.
\newblock {\em Comment. Math. Univ. Carolinae}, 15:589--602, 1974.

\bibitem[Ad{\'{a}}02]{A}
J.~Ad{\'{a}}mek.
\newblock {Final coalgebras are ideal completions of initial algebras}.
\newblock {\em J. Logic Comput.}, 12:217--242, 2002.

\bibitem[Ad{\'{a}}19]{A1}
J.~Ad{\'{a}}mek.
\newblock {On terminal coalgebras derived from initial algebras}.
\newblock {\em Proc. CALCO 2019, LIPIcs}, 390:12:1--12:21, 2019.

\bibitem[AK79]{AK-79}
J.~Ad{\'{a}}mek and V.~Koubek.
\newblock {Least fixed point of a functor}.
\newblock {\em J. Comput. System Sciences}, 19:163--168, 1979.

\bibitem[AK95]{AK}
J.~Ad{\'{a}}mek and V.~Koubek.
\newblock {On the greatest fixed point of a set functor}.
\newblock {\em Theoret. Comput. Sci.}, 150:57--75, 1995.

\bibitem[AMSW19]{AMSW}
J.~Ad{\'{a}}mek, S.~Milius, L.~Sousa, and T.~Wissmann.
\newblock {On finitary functors and finitely presentable algebras}.
\newblock {\em Theory Appl. Categories}, 34:1134--1164, 2019.

\bibitem[AMV04]{AMV}
J.~Ad{\'{a}}mek, S.~Milius, and J.~Velebil.
\newblock {On coalgebra based on classes}.
\newblock {\em Theoret. Comput. Sci.}, 316:3--23, 2004.

\bibitem[AR94]{AR}
J.~Ad{\'{a}}mek and J.~Rosick{\'{y}}.
\newblock {\em {Locally presentable and accessible categoreis}}.
\newblock Cambridge University Press, 1994.

\bibitem[AT90]{AT}
J.~Ad{\'{a}}mek and V.~Trnkov{\'{a}}.
\newblock {\em {Automata and Algebras in Categories}}.
\newblock Kluwer Acad. Publ., London, 1990.

\bibitem[Bar93]{B}
M.~Barr.
\newblock {Terminal coalgebras in well-founded set theory}.
\newblock {\em Theoret. Comput. Sci.}, 114:299--315, 1993.

\bibitem[Bau76]{Baum}
J.E. Baumgartner.
\newblock {Almost disjoint sets, the dense set problem and the partition
  calculus}.
\newblock {\em Annals Math. Logic}, 10:401--409, 1976.

\bibitem[Fre70]{F}
P.~Freyd.
\newblock {Algebraically complete categories}.
\newblock {\em {Lecture Notes in Math.}}, 1488:95--104, 1970.

\bibitem[Jec78]{J}
T.~Jech.
\newblock {\em {Set theory}}.
\newblock Academic Press, 1978.

\bibitem[Rut00]{R}
J.~Rutten.
\newblock {Universal coalgebra}.
\newblock {\em Theoret. Comput. Sci.}, 249:3--80, 2000.

\bibitem[SP82]{SP}
M.B. Smyth and G.D. Plotkin.
\newblock {The category-theoretic solution of recursive domain equations}.
\newblock {\em SIAM J. Comput.}, 11:761--788, 1982.

\bibitem[TAKR74]{TAKR}
V.~Trnkov{\'{a}}, J.~Ad{\'{a}}mek, V.~Koubek, and J.~Reiterman.
\newblock {Free algebras, input processes and free monads}.
\newblock {\em Comment. Math. Univ. Carolinae}, 15:589--602, 1974.

\bibitem[Tar28]{Ta}
A.~Tarski.
\newblock {Sur la de composition des ensembles en sous-ensembles pr{\`{e}}sque
  disjoint}.
\newblock {\em Fund. Math.}, 12:188--205, 1928.

\bibitem[Trn71]{T}
V.~Trnkov{\'{a}}.
\newblock {On descriptive classification of set-functors I, II.}
\newblock {\em Comment. Math. Univ. Carolinae}, 12:143--175 and 345--357, 1971.

\bibitem[Wor05]{W}
J.~Worrell.
\newblock {On the final sequence of a finitary set functor}.
\newblock {\em Theoret. Comput. Sci.}, 338:184--199, 2005.

\end{thebibliography}


 \end{document}